\documentclass[10pt]{article}

\usepackage{amsmath}
\usepackage{amsfonts}
\usepackage{amssymb}
\usepackage{amsrefs}
\usepackage{amsthm}
\usepackage{hyperref}
\usepackage{mathrsfs}
\usepackage{algorithm}
\usepackage{algpseudocode}
\usepackage{subcaption}
\usepackage{tikz}
\usetikzlibrary{math}
\usetikzlibrary{calc}
\usetikzlibrary{arrows.meta}
\usetikzlibrary{backgrounds}

\newtheorem{theorem}{Theorem}
\newtheorem{lemma}[theorem]{Lemma}

\newtheorem{conjecture}[theorem]{Conjecture}

\newtheorem{claim}[theorem]{Claim}

\theoremstyle{definition}
\newtheorem{definition}{Definition}

\title{Edge isoperimetry of lattices}
\author{Cameron Strachan \and Konrad Swanepoel}
\date{}

\begin{document}
\maketitle
\begin{abstract}
    We present two results related to an edge-isoperimetric question for Cayley graphs on the integer lattice asked by Ben Barber and Joshua Erde [Isoperimetry of Integer Lattices, Discrete Analysis \textbf{7} (2018)].
    For any (undirected) graph $G$, the edge boundary of a subset of vertices $S$ is the number of edges between $S$ and its complement in $G$.
    Barber and Erde asked whether for any Cayley graph on $\mathbb{Z}^d$, there is always an ordering of $\mathbb{Z}^d$ such that for each $n$, the first $n$ terms minimize the edge boundary among all subsets of size $n$.

    First, we present an example of a Cayley graph $G_d$ on $\mathbb{Z}^d$ (for all $d\geq 2$) for which there is no such ordering. Furthermore, we show that for all $n$ and any optimal $n$-vertex subset $S_n$ of $G_d$, there is no infinite sequence $S_n\subset S_{n+1}\subset S_{n+2}\subset\cdots$ of optimal sets $S_i$, where $|S_i|=i$ for $i\geq n$.
    This is to be contrasted with the positive result in $\mathbb{Z}^1$ shown by Joseph Briggs and Chris Wells [arXiv:2402.14087].
    
    Our second result is a positive example for the unit-length triangular lattice (which is isomorphic to $\mathbb{Z}^2$) where two vertices are connected by an edge if their distance is $1$ or $\sqrt{3}$. We show that this graph has such an ordering.
    This is the most complicated example known to us of a two-dimensional Cayley graph for which an ordering exists.
\end{abstract}

\section{Introduction}

\begin{definition}\label{def:edgeboundary}
    Given a graph \(G\), the \emph{edge boundary} of \(S\subseteq V(G)\) is
    \[
    \partial(S):=|\{ uv\in E(G): u\in S, v\notin S\}|.
    \]

\end{definition}

The \emph{edge isoperimetric problem} (EIP) of a graph \(G\) is, for a given \(n\), to minimize \(\partial(S)\) over all \(S\subseteq V(G)\) where \(|S|=n\). We call such minimizing sets \emph{solutions to the EIP of }\(G\). This classical problem has been extensively studied since the 1960s (see \cite{Harper2004}). Although it is NP-hard in general, some special cases are known. One aspect that has received particular attention is whether nested solutions exist. A \emph{nested solution} for the EIP of $G$ is an ordering \(v_1,v_2,\dots \) of the vertex set $V(G)$ such that for each \(n\), the set \(\{v_1,v_2,\dots, v_n\}\) is a solution to the EIP of $G$.

One of the first cases of the EIP that has been solved is the $d$-dimensional cube graph, which has nested solutions, and where the optimal shapes include subcubes \cites{Harper1964, Lindsey1964, Bernstein1967, Hart1976}.

For \(p=1,\infty\), denote by \((\mathbb{Z}^d, \ell_p) \) the graph with vertex set \(\mathbb{Z}^d\) where pairs of vertices have an edge if their \(\ell_p\) distance is \(1\). Bollob\'as and Leader \cite{BL91} solved the EIP for \((\mathbb{Z}^d, \ell_1)\), and proved that the solutions include cubes. Moreover, they showed that this graph has nested solutions.

Bollob\'as and Leader \cite{BL91} also considered the EIP on finite grids \(\{1,2,\dots, k\}^d\), considered as induced subgraphs of \((\mathbb{Z}^d, \ell_1)\). It turned out that there are two types of solutions: cubes if $n$ is small relative to the size of the grid and half-grids for large $n$. Furthermore, they showed the transition between these two types of solutions is not smooth, giving the first example of a graph without nested solutions. 

If \(G\) is an undirected \(k\)-regular graph, for any \(S\subseteq V(G)\) we have \(|E(G[S])|=\frac{k|S|-\partial(S)}{2}\). If \(G\) is a directed \(k\)-regular graph, then for any \(S\subseteq V(G)\) we have \(|E(G[S])|=k|S|-\partial(S)\). Thus, for regular graphs the problem of minimizing \(\partial(S)\) over all subsets with size \(n\) is the same as maximizing \(|E(G[S])|\) over all subsets of size \(n\).

In terms of this formulation, Brass \cite{Brass96} solved the EIP of \((\mathbb{Z}^2,\ell_\infty)\), where the optimal shapes include certain octagons. Additionally, he showed \((\mathbb{Z}^2,\ell_\infty)\) has nested solutions. For \(d\geq 3\), the EIP of \((\mathbb{Z}^d,\ell_\infty)\) remains open. 

Let \(g_1=(1,0)\) and \(g_2=(1/2,\sqrt{3}/2)\). The \emph{triangular lattice} is the set
\[
\Lambda:=\{ mg_1+ng_2: m,n\in \mathbb{Z}\}.
\]
We can turn $\Lambda$ into a graph by joining a pair of vertices if their Euclidean distance is $1$.
For this graph, the EIP is solved \cite{Harborth1974} (see also \cites{HH1976, Harper2001}) with solutions that include regular hexagons. Again, the graph has nested solutions.
This graph is isomorphic to $\mathbb{Z}^2$, where two vertices are joined if their difference is in $\{(\pm1,0),(0,\pm1), \pm(1,1)\}$, hence it can be thought of as a graph between \((\mathbb{Z}^2,\ell_1)\) and \((\mathbb{Z}^2,\ell_\infty)\).

The above examples are all special cases of Cayley graphs on the group $\mathbb{Z}^d$.

\begin{definition}\label{def:Cayley}
    Let \(U\) be a finite set that generates \(\mathbb{Z}^d\) as a group and does not contain the identity. The (directed) Cayley graph $\mathbb{Z}^d_U$ is the graph on  the vertex set \(\mathbb{Z}^d\) where $(u,v)$ is an edge whenever $v-u\in U$. When \(U\) is symmetric (that is, $-u\in U$ for all $u\in U$), we consider \(\mathbb{Z}^d_U\) to be undirected.
\end{definition}

Given a generating set $U$ of \(\mathbb{Z}^d\), let \(Z\subseteq \mathbb{R}^d\) be the zonotope $\sum_{u\in U}[0,u]$ generated by the line segments \([0,u]\), \(u\in U\). Barber and Erde \cite{BE2018} showed that the edge boundary of \(tZ\cap \mathbb{Z}^d\) for large \(t\), asymptotically approximates the edge boundary of solutions to the EIP of \(\mathbb{Z}^d_U\). Barber, Erde, Keevash and Roberts \cite{BEKR2023} showed that additionally, \(tZ\cap \mathbb{Z}^d\) asymptotically approximates the shape of solutions to the EIP of \(\mathbb{Z}^d_U\).

Barber and Erde \cite{BE2018} asked if every Cayley graph of \(\mathbb{Z}^d\) has nested solutions. Despite the positive examples already given, Briggs and Wells \cite{BW2024} gave counterexamples for the case $d=1$. On the other hand, they also gave a partial positive answer: for any Cayley graph of \(\mathbb{Z}\), there exists an \(m\in \mathbb{N}\) and an ordering \(v_1,v_2\dots\) of \(\mathbb{Z}\) such that for any \(n\geq m \), the set \(\{v_1,v_2,\dots, v_n\}\) is a solution to the EIP. In other words, they showed that any Cayley graph on \(\mathbb{Z}\) has nested solutions starting at a sufficiently large size.

We give a negative answer to the question of Barber and Erde for all \(d\geq 2\) by giving an explicit example of a Cayley graph without nested solutions. Furthermore, we show this example does not have nested solutions regardless of any starting point. 
Thus, in dimensions $2$ and higher there are stronger counterexamples than in $\mathbb{Z}^1$.

\begin{theorem}\label{the:}

The EIP for $\mathbb{Z}^d_U$, where $U$ is the generating set $\{\pm e_i : i=1,\dots,d\}\cup\{\pm 2e_1\}$  of $\mathbb{Z}^d$, does not have nested solutions starting at any size.
In other words, for all $n$ and each $n$-element subset $S_n$ of $\mathbb{Z}^d$ for which $\partial(S_n)$ is the minimum among all $n$-element subsets of $\mathbb{Z}^d_U$, there does not exist a sequence $S_n\subset S_{n+1}\subset S_{n+2}\subset\cdots$ of $i$-element subsets $S_i$ of $\mathbb{Z}^d$, such that for each $i\geq n$, $\partial(S_i)$ is the minimum among all $i$-element subsets of $\mathbb{Z}^d_U$.
\end{theorem}

Our second result is a solution of the EIP for another Cayley graph on \(\mathbb{Z}^2\) with nested solutions.
The generating set for this graph is $U=\{(\pm1,0),(0,\pm1),\\ \pm(1,1),\pm(1,-1),\pm(-1,2),\pm(-2,1)\}$.
Thus, it contains \((\mathbb{Z}^2,\ell_\infty)\) as a subgraph.
In fact, it is more suitable to consider this to be the graph on the triangular lattice with edges for all pairs at distance $1$ or $\sqrt{3}$.
As a Cayley graph on \(\Lambda\), the generating set is depicted in Figure~\ref{fig:generatingset}.
We denote it by $\Lambda_U$.

\begin{figure}
    \centering
    \begin{tikzpicture}[>=Stealth,scale=1.3]
    \def\points{(g1)/$g_1$/0,($(g1)+(g2)$)/$g_1+g_2$/30,(g2)/$g_2$/60,($2*(g2)-(g1)$)/$2g_2-g_1$/90,($(g2)-(g1)$)/$g_2-g_1$/120, ($-2*(g1)+(g2)$)/$-2g_1+g_2$/150,($-1*(g1)$)/$-g_1$/180,($-1*(g1)-(g2)$)/$-g_1-g_2$/210,($-1*(g2)$)/$-g_2$/240,($-2*(g2)+(g1)$)/$g_1-2g_2$/270,($-1*(g2)+(g1)$)/$g_1-g_2$/300, ($2*(g1)-(g2)$)/$2g_1-g_2$/330}
    \clip (-2.55,-2.2) rectangle (2.55,2.2);
    \coordinate (g1) at (1,0);
    \coordinate (g2) at (60:1);
    \foreach \x in {-3,...,3} 
        {
        \foreach \y in {-3,...,3}
        {
        \coordinate (cur) at ($\x*(g1)+\y*(g2)$);
        \filldraw[draw=black] (cur) circle [radius=0.025];
        }
        }
    \foreach \p/\ptext/\pang in \points {
        \draw[->,thick] (0,0) -- \p;
        \node at \p [label={[label distance=-1mm]\pang:\scriptsize{\ptext}}]  {};
    }
    \end{tikzpicture}
    \caption{Generating set of Theorem \ref{the:e}}
    \label{fig:generatingset}
\end{figure}
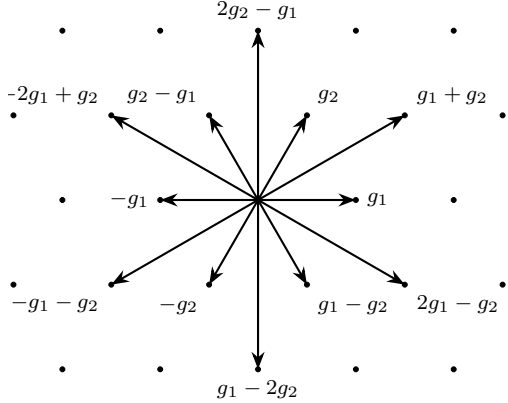

\begin{theorem}\label{the:e}
        Let \(\Lambda_U\) be the undirected Cayley graph with vertex set \( \Lambda\) and symmetric generating set
    \(U=\{ \pm g_1, \pm(g_1+g_2), \pm g_2, \pm(2g_2-g_1), \pm(g_2-g_1), \pm(g_2-2g_1)\}\), where \(g_1=(1,0)\) and \(g_2=(1/2,\sqrt{3}/2)\).
    The maximum number of edges of a subgraph of \(\Lambda_U\) (necessarily induced) with \(n\geq 3\) vertices is
    \[
    e(n):=\begin{cases}
        6n-4\sqrt{6n-6}\text{ if }n=24k^2-24k+7 \text{ for some }k \in \mathbb{N} \\
        \lfloor6n-\sqrt{96n-63}\rfloor \text{ otherwise}.
        \end{cases}
    \]
    Additionally, the EIP of \(\Lambda_U\) has nested solutions.
\end{theorem}

The first few values of \(n\) where \(e(n)=6n-4\sqrt{6n-6}\) are \(n=7,55, 151, 295,\\ 487\) and \(727\). In Figure~\ref{fig:extremal} we depict the unique (up to translation) extremal subgraph of $\Lambda_U$ with \(55\) vertices.

The subgraphs of $\Lambda_U$ with $n$ vertices and $e(n)$ edges are candidate extremal graphs for a problem of Erd\H{o}s and Vesztergombi \cite{Csizmadia} on the maximum number of occurrences of the smallest and second smallest distances in a set of $n$ points in the plane.
Let $S$ be a set of $n$ points in the plane, and denote by $m_1(S)$ and $m_2(S)$ the number of occurrences of the smallest and second smallest distance in $S$.
Let $f(n)$ be the maximum value of $m_1(S)+m_2(S)$, where the maximum is taken over all sets $S$ of $n$ points.
Vesztergombi \cite{Vesztergombi} showed that $f(n)\leq 6n$.
(See also Csizmadia \cite{Csizmadia} for further results.)
Theorem~\ref{the:e} implies that $f(n)\geq e(n)$,
with the lower bound being given by subsets of the triangular lattice, with smallest distance $1$ and second smallest distance $\sqrt{3}$.
\begin{conjecture}
For any sufficiently large $n$, $f(n)=e(n)$, with the only sets $S$ of $n$ points attaining $f(n)=m_1(S)+m_2(S)$ being geometrically similar to the extremal sets on the triangular lattice.
\end{conjecture}

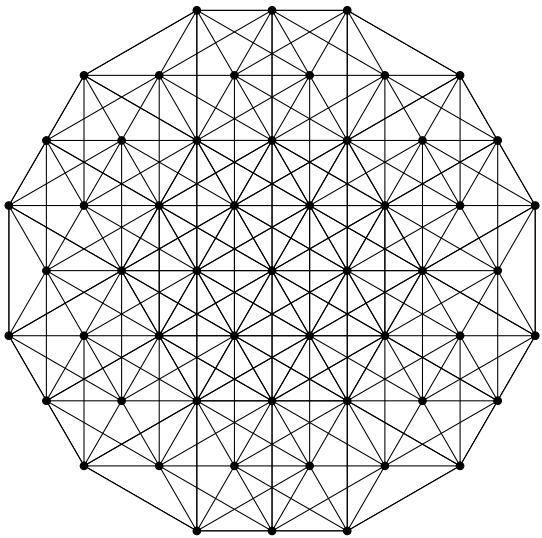
\begin{figure}
    \centering
\begin{tikzpicture}
\def\points{(e1),(e2),($(e1)+(e2)$),($2*(e2)-(e1)$),($(e2)-(e1)$), ($-2*(e1)+(e2)$),($-1*(e1)$),($-1*(e2)$),($-1*(e1)-(e2)$),($-2*(e2)+(e1)$),($-1*(e2)+(e1)$), ($2*(e1)-(e2)$)}
\def\bpoints{($3*(e2)+(e3)-(e1)$),($3*(e2)+(e3)$),($3*(e2)+(e1)$),($3*(e2)+(e1)-(e3)$),($3*(e1)+(e2)$),($3*(e1)-(e3)$),($3*(e1)-(e3)-(e2)$),($-3*(e3)+(e1)$),($-3*(e3)-(e2)$),($-3*(e3)-(e2)-(e1)$),($-3*(e2)-(e3)$),($-3*(e2)-(e1)$),($-3*(e2)-(e1)+(e3)$),($-3*(e1)-(e2)$),($-3*(e1)+(e3)$),($-3*(e1)+(e3)+(e2)$),($3*(e3)-(e1)$),($3*(e3)+(e2)$)}
\def\ipoints{(0,0), (e1), (e2), (e3), ($-1*(e1)$), ($-1*(e2)$), ($-1*(e3)$),
            ($(e1)+(e1)$), ($(e1)+(e2)$), ($(e1)-(e3)$),
            ($-1*(e1)-(e1)$), ($-1*(e1)-(e2)$), ($-1*(e1)+(e3)$),
            ($(e2)+(e2)$), ($(e2)+(e3)$), ($(e2)+(e3)-(e1)$),
            ($-1*(e2)-(e2)$), ($-1*(e2)-(e3)$), ($-1*(e2)-(e3)+(e1)$)}
\coordinate (e1) at (1,0);
\coordinate (e2) at (60:1);
\coordinate (e3) at (120:1);
\begin{scope}[on background layer]
\clip ($3*(e2)+(e3)$)-- ($3*(e2)+(e1)$) -- ($3*(e1)+(e2)$) --($3*(e1)-(e3)$) -- ($-3*(e3)+(e1)$) -- ($-3*(e3)-(e2)$) -- ($-3*(e2)-(e3)$) -- ($-3*(e2)-(e1)$) -- ($-3*(e1)-(e2)$) -- ($-3*(e1)+(e3)$) -- ($3*(e3)-(e1)$) -- ($3*(e3)+(e2)$);
\foreach \x in {-4,...,4} 
    {
    \foreach \y in {-4,...,4}
    {
    \coordinate (cur) at ($\x*(e1)+\y*(e2)$);
    \filldraw[draw=black] (cur) circle [radius=0.05];
    }
    }
\foreach \b in \bpoints {
    \foreach \p in \points{
        \draw \b -- +\p;
        }
    }
\end{scope}
\foreach \i in  \ipoints {
    \foreach \p in \points{
    \draw \i -- +\p;
    }
    }
\draw ($3*(e1)$) -- ($3*(e2)$) -- ($3*(e3)$) -- ($-3*(e1)$)-- ($-3*(e2)$) -- ($-3*(e3)$) -- cycle;
\draw ($3*(e2)+(e3)$)-- ($3*(e2)+(e1)$) -- ($3*(e1)+(e2)$) --($3*(e1)-(e3)$) -- ($-3*(e3)+(e1)$) -- ($-3*(e3)-(e2)$) -- ($-3*(e2)-(e3)$) -- ($-3*(e2)-(e1)$) -- ($-3*(e1)-(e2)$) -- ($-3*(e1)+(e3)$) -- ($3*(e3)-(e1)$) -- ($3*(e3)+(e2)$) -- cycle;
\draw ($-2*(e1)+(e3)$)--($-2*(e1)-(e2)$);
\draw ($-2*(e2)-(e1)$)--($-2*(e2)-(e3)$);
\draw ($-2*(e3)+(e1)$)--($-2*(e3)-(e2)$);
\draw ($2*(e1)+(e2)$)--($2*(e1)-(e3)$);
\draw ($2*(e2)+(e1)$)--($2*(e2)+(e3)$);
\draw ($2*(e3)+(e2)$)--($2*(e3)-(e1)$);

\foreach \p in \bpoints 
    {
    \filldraw[draw=black] \p circle [radius=0.05];
    }

\end{tikzpicture}
    \caption{The extremal subgraph of \(\Lambda_U\) with \(24k^2-24k+7\) vertices (\(k=2\))}
    \label{fig:extremal}
\end{figure}

\section{Non-existence of an ordering}

\begin{proof}[Proof of Theorem~\ref{the:}]
Let $e_1,\dots,e_d$ denote the standard unit basis of $\mathbb{Z}^d$, and $U=\{\pm e_i : i=1,\dots,d\}\cup\{\pm 2e_1\}$.
Then in the graph $\mathbb{Z}^d_U$, two vertices $x=(x_1,x_2,\dots,x_d)$ and $y=(y_1,y_2,\dots,y_d)$ are joined when either $\|x-y\|_1=1$ or $|x_1-y_1|=2$.

Let $n\in\mathbb{N}$ be arbitrary, and let $S$ be an $n$-element subset of $\mathbb{Z}^d$ for which the number of edges of the subgraph $\mathbb{Z}^d_U[S]$ induced by $S$ is the maximum among all $n$-element subsets of $\mathbb{Z}^d$.
We will show that there does not exist a sequence $S_i$, $i\geq n$, of subsets of $\mathbb{Z}^d$ such that $S=S_n$, $|S_i|=i$, $S_i\subset  S_{i+1}$ for all $i\geq n$, and each $S_i$ maximizes the number of edges in the subgraph $\mathbb{Z}^d_U[S_i]$.

Suppose then that there is such a sequence.
Without loss of generality we may assume that $n > 2^{d-1}$.
Let $m_j=\min\{x_j:(x_1,x_2,\dots,x_d)\in S\}$ and $M_j=\max\{x_j:(x_1,x_2,\dots,x_d)\in S\}$ for each $j=2,\dots,d$.
Let $C=\mathbb{Z}\times\prod_{j=2}^d\{m_j,\dots,M_j\}$.
We will prove by induction that $S_i\subset C$ for all $i \geq n$, which will subsequently lead to a contradiction.

The basis case is trivial, since $S_n=S\subset C$ by the definition of $C$.
As an induction hypothesis we assume that for fixed $m\geq n$, $S_m\subset C$.
We have to show that $S_{m+1}\subset C$.

Note that for any $x\in\mathbb{Z}^d\setminus C$, there is at most one edge from $x$ to an element of $C$.
Thus, to show that $S_{m+1}\subset C$, it is sufficient to prove that there exists $x\in C\setminus S_m$ such that $x$ has at least two neighbours in $S_m$, since this will imply that the number of neighbours between the point $x$ in $S_{m+1}\setminus S_m$ and $S_m$ has to be at least $2$, hence cannot be outside $C$.
We prove this by considering three mutually exclusive and exhaustive possibilities for the induced subgraph on $S_m$.

First, if $\mathbb{Z}^d_U[S_m]$ contains an edge $xy$ where $y-x= e_1$, then we may, by repeatedly adding $e_1$ to both $x$ and $y$, obtain $x,y\in S_m$ such that $y-x=e_1$, and $z:=y+e_1\notin S_m$.
However, then $z$ has at least the two neighbours $x$ and $y$ in $\mathbb{Z}^d_U$.

Second, if $\mathbb{Z}^d_U[S_m]$ does not contain any edge $xy$ where $y-x= e_1$, but does contain an edge $xy$ where $y-x=2e_1$, then $z:=x+e_1 =y-e_1\notin S_m$ has neighbours $x$ and $y$ in $\mathbb{Z}^d_U$.

Third, if $\mathbb{Z}^d_U[S_m]$ does not contain any edge $xy$ where $x-y\in\{\pm e_1,\pm 2e_1\}$, then all its edges are in the directions $\{\pm e_2,\dots,\pm e_d\}$.
In this case the subgraph decomposes into connected components that can be embedded into $(\mathbb{Z}^{d-1},\ell_1)$, which can be considered to be the subgraph induced by the coordinate hyperplane $x_1=0$.
However, it follows from \cite{BL91} that no optimal set for $(\mathbb{Z}^{d},\ell_1)$ can be contained in a coordinate hyperplane, unless $m\leq 2^{d-1}$.
Thus, if $m > 2^{d-1}$, then $S_m$ is not optimal for this supergraph of $(\mathbb{Z}^{d},\ell_1)$ either.

This finishes the induction step.
We then have that all $S_m$, $m\geq n$, are contained in $C$.
In order to obtain a contradiction, we now switch to considering the edge boundary of $S_m$, which we bound from below by using the Loomis--Whitney inequality \cite{LW1949}.
This will then be compared to a construction with a smaller edge boundary for large $m$, which will give the required contradiction.

Let $\pi_i\colon\mathbb{Z}^d\to\mathbb{Z}^{d-1}$, $i=1,\dots,d$, be the projection that deletes the $i$-th coordinate.
For each $i=1,\dots,d$, let $P_i=\pi(S_m)$.
Note that for each $y\in P_1$, the element $x$ in $S_m\cap\pi_1^{-1}(y)$ with smallest 1st coordinate has that $x-e_1,x-2e_1\notin S_m$, hence contributes $2$ to $\partial(S_m)$.
Similarly, the element in $S_m\cap\pi_1^{-1}(y)$ with largest 1st coordinate contributes an additional $2$ to $\partial(S_m)$.
Thus there is a contribution of $4|P_1|$ to $\partial(S_m)$.
In the same way, for each $i=2,\dots,d$, there is a contribution of $2|P_i|$ to $\partial(S_m)$.
Thus $\partial(S_m)\geq 4|P_1|+\sum_{i=2}^d2|P_i|$.
The Loomis--Whitney inequality states that $\prod_{i=1}^d|P_i|\geq |S_m|^{d-1}=m^{d-1}$.
Since $P_1\subseteq\prod_{j=2}^d\{m_j,\dots,M_j\}$, its cardinality is bounded: $|P_1|\leq\prod_{j=2}^d(M_j-m_j+1)=:M$.
It follows that $\prod_{i=2}^d|P_i|\geq m^{d-1}/M$, and applying the AM-GM inequality, we obtain $\frac1{d-1}\sum_{i=2}^d|P_i|\geq(\prod_{i=2}^d|P_i|)^{1/(d-1)}\geq m/M^{1/(d-1)}$, hence
$\partial(S_m)\geq cm$, where $c=(d-1)/M^{1/(d-1)}$ is a constant.

On the other hand, since $S_m$ is optimal, it cannot have a larger edge boundary than any other set of $m$ elements.
If we choose $m=k^d$ for some $k$, then we have for $S'=\prod_{i=1}^d\{1,\dots,k\}$ that $\partial(S')=2(d+1)k^{d-1}=2(d+1)m^{1-1/d}$.
Therefore, $cm\leq\partial(S_m)\leq\partial(S')=2(d+1)m^{1-1/d}$ which is a contradiction for large $m$.
\end{proof}

\section{Proof outline of Theorem \ref{the:e}}

Let \(\Lambda_U\) be the Cayley graph of \(\Lambda\) with generating set \(U\) stated in Theorem \ref{the:e}. We refer to edges in \(\Lambda_U\) as \emph{short edges} if they have unit length, and \emph{long edges} if they have length \(\sqrt{3}\). Since \(\Lambda_U\) is a \(12\)-regular graph, we have for any \(S\subseteq \Lambda\) with \(n\) vertices
\begin{equation}\label{eq:eindef}
    |E(\Lambda_U[S])|=6n-\frac{\partial(S)}{2}.
\end{equation}

To prove Theorem \ref{the:e} we first show in Section \ref{sec:upbound} that any subgraph of \(\Lambda_U\) with \(n\geq 3\) vertices has at most \(e(n)\) edges. In Section \ref{sec:seq}, we then give an ordering \(v_1,v_2,\dots\) of \(\Lambda\) such that for each \(n\), the graph \(\Lambda_U[\{v_1,v_2,\dots, v_n\}]\) has \(e(n)\) many edges. This ordering, together with the upper bound, proves that \(\Lambda_U\) has nested solutions. 

To show that \(e(n)\) is an upper bound for the number of edges of an \(n\)-vertex subgraph of \(\Lambda_U\), we first show this upper bound for a particular class of subgraphs of \(\Lambda_U\) that can be thought of as (possibly degenerate) completely filled-up lattice $12$-gons.
We define these as follows.

\begin{definition}\label{def:p(g)}
    For any finite \(S\subseteq \Lambda\), the \emph{hull} of \(S\), denoted as \(\operatorname{hull}(S)\subseteq \Lambda\), is the intersection of the \(12\) supporting half-planes of \(S\) parallel to an element of \(U\). Denote \(\mathscr{P}=\{\Lambda_U[\operatorname{hull}(S)]: S\subseteq \Lambda, |S|<\infty\}\).
\end{definition}

\begin{figure}
    \centering
\begin{tikzpicture}
\coordinate (e1) at (1,0);
\coordinate (e2) at (60:1);
\coordinate (e3) at (120:1);
\begin{scope}[on background layer]
\clip (-3.4,-3) rectangle (3.4,3);
\foreach \x in {-4,...,4} 
    {
    \foreach \y in {-4,...,4}
    {
    \coordinate (cur) at ($\x*(e1)+\y*(e2)$);
    \filldraw[draw=black] (cur) circle [radius=0.025];
    }
    }

\draw[gray, dashed] ($-5*(e2)$) -- ($5*(e3)$);
\draw[gray, dashed] ($5*(e2)$) -- ($-5*(e3)$);
\draw[gray, dashed] ($3*(e3)+5*(e2)$) -- ($3*(e3)-5*(e2)$);
\draw[gray, dashed] ($-2*(e3)+(e1)+5*(e2)$) -- ($-2*(e3)+(e1)-5*(e2)$);

\draw[gray, dashed] ($2*(e2)+(e1)+5*(e3)$) -- ($2*(e2)+(e1)-5*(e3)$);
\draw[gray, dashed] ($-2*(e1)-(e2)+5*(e3)$) -- ($-2*(e1)-(e2)-5*(e3)$);

\draw[gray, dashed] ($-3*(e2)-3*(e1)$) -- ($-3*(e3)+3*(e1)$);
\draw[gray, dashed] ($3*(e2)+3*(e1)$) -- ($3*(e3)-3*(e1)$);
\draw[gray, dashed] ($3*(e3)+5*(e1)+5*(e2)$) -- ($3*(e3)-5*(e1)-5*(e2)$);
\draw[gray, dashed] ($-2*(e3)-(e2)+5*(e1)+5*(e2)$) -- ($-2*(e3)-(e2)-5*(e1)-5*(e2)$);
\draw[gray, dashed] ($-2*(e2)-(e1)-5*(e1)+5*(e3)$) -- ($-2*(e2)-(e1)+5*(e1)-5*(e3)$);
\draw[gray, dashed] ($2*(e2)+(e3)-5*(e1)+5*(e3)$) -- ($2*(e2)+(e3)+5*(e1)-5*(e3)$);

\end{scope}
\def\points{(0,0), (e3), ($-1*(e1)$), ($-1*(e3)$), ($2*(e1)$), ($(e1)+(e2)$), ($2*(e2)$), ($(e2)+(e3)$), ($-2*(e1)$), ($-1*(e1)-(e2)$), ($-2*(e2)$), ($2*(e1)-(e3)$), ($2*(e2)+(e1)$), ($2*(e2)+(e3)$), ($2*(e3)+(e2)$), ($-2*(e1)+(e3)$), ($-2*(e1)-(e2)$), ($-2*(e2)-(e1)$), ($-2*(e2)-(e3)$), ($-2*(e2)-(e3)+(e1)$), ($3*(e3)$)}

\def\rpoints{(e1),(e2),($-1*(e2)$), ($(e1)-(e3)$),($(e3)-(e1)$),($2*(e1)-(e2)$),($-2*(e3)$),($-1*(e1)-2*(e3)$),($(e1)-2*(e3)$),($2*(e3)$),($-1*(e1)+2*(e3)$),
($2*(e1)+(e2)$)}

\foreach \p in \points {
    \filldraw[draw=black] \p circle [radius=0.075];
    }
\foreach \p in \rpoints {
    \filldraw[draw=black, fill=red] \p circle [radius=0.075];
    }
\end{tikzpicture}
    \caption{The hull of the black points is the set of black and red points}
    \label{fig:Hull}
\end{figure}
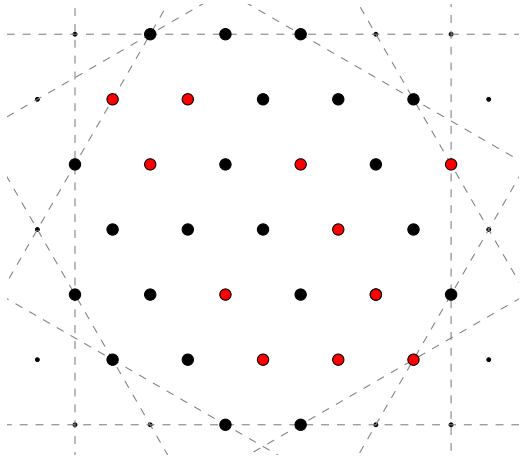

Figure~\ref{fig:Hull} depicts an example of \(\operatorname{hull}(S)\) for a particular \(S\subseteq \Lambda\). We use induction on \(n\) to show that any \(P\in \mathscr{P}\) with \(n\) vertices has at most \(e(n)\) edges. The base case of this induction goes up to $n=30$ and involves a Python computation (Algorithm~\ref{alg:basecase}). In the inductive step, we rearrange the points of $P$ to make $P$ ``rounder''. This potentially leads to constructing a new $P^*\in\mathscr{P}$ with more vertices, where we can either apply the inductive hypothesis after removing its boundary, or for a certain range of polygons very close to optimal, we need to do some exact symbolic computations using sympy (Algorithm~\ref{alg:IS}).

Once we have shown the upper bound for all sets in \(\mathscr{P}\), the general case is straightforward.
We again use induction on \(n\) to show that any subgraph of \( \Lambda_U\) with \(n\) vertices has at most \(e(n)\) edges. During this process we will prove additionally that when \(n=24k^2-24k+7\) for some \(k\in \mathbb{N}\), the \(n\)-vertex subgraph of \(\Lambda_U\) with \(e(n)\) edges is unique up to a translation. This enables us to define an ordering of \(\Lambda\) by interpolating between the unique extremal subgraph $P_k$ when \(n=24k^2-24k+7\) for some \(k\in \mathbb{N}\), and the unique extremal subgraph $P_{k+1}$ when \(n=24(k+1)^2-24(k+1)+7\).
This interpolation entails finding a specific sequence of length $48$ of the $12$ orientations of sides that have to be added to $P_k$ to get to $P_{k+1}$ (Table~\ref{table}).
This sequence is found by using Breadth-First Search in an auxiliary directed graph that represents all ways of adding sides.

\section{The upper bound for Theorem \ref{the:e}}\label{sec:upbound}
\subsection{The upper bound for polygons}

We start this section by uniquely associating, up to a translation, each \(P\in \mathscr{P}\) with a set of \(12\) parameters.
The boundary of \(P\) is a convex polygon with \(12\) possibly degenerate sides, each parallel to an element of \(U\). 
Direct each side of \(P\) in a counterclockwise manner.

\begin{definition}\label{def:uiti}
    Let \(u_1\) denote the number of edges in the side of \(P\) in the direction \(g_1\) (i.e. the length of the side of $P$ in the direction of $g_1$). Let \(t_1\) denote the number of edges in the side of \(P\) in the direction \(g_1+g_2\) (i.e. the length divided by $\sqrt{3}$ of the side of $P$ in the direction of $g_1+g_2$). Let \(u_2\) denote the number of edges in the side of \(P\) in the direction \(g_2\). Let \(t_2\) denote the number of edges in the side of \(P\) in the direction \(2g_2-g_1\). Continue in this way counterclockwise around the boundary of \(P\), alternating between \(u_i\) and \(t_i\). Figure~\ref{fig:uiti} depicts an example of which sides of \(P\) correspond to which parameters.
\end{definition}

\begin{figure}
    \centering
\begin{tikzpicture}[>=Stealth]
\coordinate (e1) at (1,0);
\coordinate (e2) at (60:1);
\coordinate (e3) at (120:1);
\begin{scope}[on background layer]
\clip (-3.4,-3) rectangle (3.4,3);
\filldraw[gray!20!white] ($-2*(e2)-(e3)$)--++(e1)--++($(e1)+(e2)$)--++(e2)--++($(e2)+(e3)$)--++(e3)--++($(e3)-(e1)$)--++($-1*(e1)$)--++($-1*(e1)-(e2)$)--++($-1*(e2)$)--++($-1*(e2)-(e3)$)--++($-1*(e3)$)--cycle;
\foreach \x in {-4,...,4} 
    {
    \foreach \y in {-4,...,4}
    {
    \coordinate (cur) at ($\x*(e1)+\y*(e2)$);
    \filldraw[draw=black] (cur) circle [radius=0.025];
    }
    }
\end{scope}
\draw[->] ($-2*(e2)-(e3)$)--++(e1) node[midway, label={[label distance=-1.5mm]90:$u_1$}] {};
\draw[->] ($-2*(e3)-(e2)$)--++($(e1)+(e2)$) node[midway, label={[label distance=-2.5mm]120:$t_1$}] {};
\draw[->] ($-2*(e3)+(e1)$)--++(e2) node[midway, label={[label distance=-2.5mm]150:$u_2$}] {};
\draw[->] ($2*(e1)-(e3)$)--++($(e2)+(e3)$)node[midway, label={[label distance=-1.5mm]180:$t_2$}] {};
\draw[->] ($2*(e1)+(e2)$)--++(e3)node[midway, label={[label distance=-2.5mm]210:$u_3$}] {};
\draw[->] ($2*(e2)+(e1)$)--++($(e3)-(e1)$) node[midway, label={[label distance=-2.5mm]240:$t_3$}] {};
\draw[->] ($2*(e2)+(e3)$)--++($-1*(e1)$)node[midway, label={[label distance=-1.5mm]270:$u_4$}] {};
\draw[->] ($2*(e3)+(e2)$)--++($-1*(e1)-(e2)$)node[midway, label={[label distance=-2.5mm]300:$t_4$}] {};
\draw[->] ($2*(e3)-(e1)$)--++($-1*(e2)$)node[midway, label={[label distance=-2.5mm]330:$u_5$}] {};
\draw[->] ($-2*(e1)+(e3)$)--++($-1*(e2)-(e3)$)node[midway, label={[label distance=-1.5mm]0:$t_5$}] {};
\draw[->] ($-2*(e1)-(e2)$)--++($-1*(e3)$)node[midway, label={[label distance=-2.5mm]30:$u_6$}] {};
\draw[->] ($-2*(e2)-(e1)$)--++($-1*(e3)+(e1)$) node[midway, label={[label distance=-2.5mm]60:$t_6$}] {};
\end{tikzpicture}
    \caption{What sides of \(P\) correspond to the parameters \(u_i\) and \(t_i\) \((i=1,\dots, 6)\) }
    \label{fig:uiti}
\end{figure}
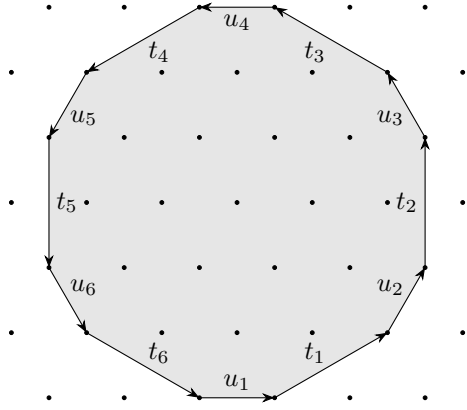
\begin{theorem}\label{the:P}
    An \(n\)-vertex \(P\in\mathscr{P}\) has at most \(\lfloor6n-\sqrt{96n-63}\rfloor\) edges, unless \(u_i=k\) and \(t_i=k-1\) for each \(i=1,\dots, 6\) and some \(k\in \mathbb{N}\), in which case \(P\) has at most \(6n-4\sqrt{6n-6}\) edges.
\end{theorem}

This statement is a stronger version of Theorem \ref{the:e} restricted to the class \(\mathscr{P}\). Before proving this theorem by induction on \(n\), first we prove three lemmas.

\begin{lemma}\label{lem:nuiti}
    For any \(n\)-vertex \(P\in \mathscr{P}\) we have
    \begin{align*}
    n& = (t_2+2t_3+t_4+u_3+u_4+1)(t_1+2t_2+t_3+u_2+u_3+1)\\
     & \mathrel{\phantom{=}} {} -\binom{t_2+t_3+u_3+1}{2}-\binom{t_5+t_6+u_6+1}{2}-\sum_{i=1}^{6}{\binom{t_i+1}{2}}.
    \end{align*}
\end{lemma}
\begin{proof}
    The intersection of the four supporting half-planes of \(P\) that are parallel to the directions $g_1$ and $g_2$ form a possibly degenerate circumscribed parallelogram of \(P\). This is depicted in blue for a particular \(P\) in Figure~\ref{fig:Circum}. 
    
    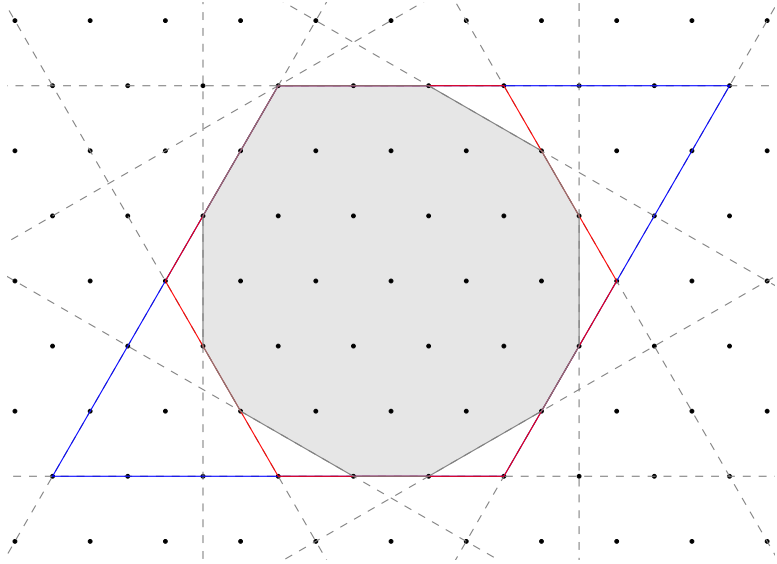
\begin{figure}
    \centering
\begin{tikzpicture}
\coordinate (e1) at (1,0);
\coordinate (e2) at (60:1);
\coordinate (e3) at (120:1);
\begin{scope}[on background layer]
\clip (-5.1,-3.7) rectangle (5.1,3.7);
\filldraw[gray!20!white] ($2*(e1)-(e3)$) --++ ($(e2)+(e3)$) --++ ($(e3)$) --++($(e3)-(e1)$) --++($-2*(e1)$) -- ++($-2*(e2)$) --++ ($-1*(e3)-(e2)$) --++ ($-1*(e3)$) --++ ($(e1)-(e3)$) --++($(e1)$) --++($(e1)+(e2)$)-- cycle;
\foreach \x in {-7,...,7} 
    {
    \foreach \y in {-7,...,7}
    {
    \coordinate (cur) at ($\x*(e1)+\y*(e2)$);
    \filldraw[draw=black] (cur) circle [radius=0.025];
    }
    }

\draw[gray, dashed] ($-3*(e2)-5*(e1)$) -- ($-3*(e3)+5*(e1)$);
\draw[gray, dashed] ($3*(e2)+5*(e1)$) -- ($3*(e3)-5*(e1)$);
\draw[gray, dashed] ($3*(e3)+10*(e2)$) -- ($3*(e3)-10*(e2)$);
\draw[gray, dashed] ($-2*(e3)+(e1)+10*(e2)$) -- ($-2*(e3)+(e1)-10*(e2)$);

\draw[gray, dashed] ($2*(e2)+(e1)+5*(e3)$) -- ($2*(e2)+(e1)-7*(e3)$);
\draw[gray, dashed] ($-2*(e1)-(e2)+6*(e3)$) -- ($-2*(e1)-(e2)-5*(e3)$);

\draw[gray, dashed] ($-5*(e2)$) -- ($5*(e3)$);
\draw[gray, dashed] ($5*(e2)$) -- ($-5*(e3)$);
\draw[gray, dashed] ($3*(e3)+5*(e1)+5*(e2)$) -- ($3*(e3)-5*(e1)-5*(e2)$);
\draw[gray, dashed] ($-2*(e3)-(e2)+5*(e1)+5*(e2)$) -- ($-2*(e3)-(e2)-5*(e1)-5*(e2)$);
\draw[gray, dashed] ($-2*(e2)-(e1)-5*(e1)+5*(e3)$) -- ($-2*(e2)-(e1)+5*(e1)-5*(e3)$);
\draw[gray, dashed] ($2*(e2)+(e3)-5*(e1)+5*(e3)$) -- ($2*(e2)+(e3)+5*(e1)-5*(e3)$);

\end{scope}

\draw[blue] ($3*(e1)+3*(e2)$) -- ($3*(e3)$) -- ($-3*(e1)-3*(e2)$) -- ($-3*(e3)$) -- cycle;
\draw[red] ($3*(e3)$) -- ($-3*(e1)$) -- ($-3*(e2)$) -- ($-3*(e3)$) -- ($3*(e1)$) -- ($3*(e2)$) -- cycle;
\draw[gray] ($2*(e1)-(e3)$) --++ ($(e2)+(e3)$) --++ ($(e3)$) --++ ($-1*(e1)+(e3)$) --++ ($-2*(e1)$) --++ ($-2*(e2)$) --++ ($-1*(e2)-(e3)$) --++ ($-1*(e3)$) --++ ($(e1)-(e3)$) --++ ($(e1)$) --++ ($(e1)+(e2)$) -- cycle;
\end{tikzpicture}
    \caption{Circumscribed parallelogram (blue) and hexagon (red) of \(P\) (gray) }
    \label{fig:Circum}
    \end{figure}
    
    The number of vertices contained in the parallelogram is easily seen to be \((t_2+2t_3+t_4+u_3+u_4+1)(t_1+2t_2+t_3+u_2+u_3+1)\). Now intersect this parallelogram with the two supporting half-planes of \(P\) parallel to \(g_2-g_1\). Figure~\ref{fig:Circum} depicts this in red. This creates a possibly degenerate hexagon by cutting off two triangular corners of the parallelogram. This removes \(\binom{t_2+t_3+u_3+1}{2}+\binom{t_5+t_6+u_6+1}{2}\) many vertices. 

    Lastly, we intersect this hexagon with the six supporting half-planes of \(P\) parallel to the long edges of \(\Lambda_U\). Figure~\ref{fig:Circum} depicts this in gray. This forms \(P\) by cutting off each corner of the hexagon. The corner that forms the side \(t_i\) has \(\binom{t_i+1}{2}\) many vertices.
\end{proof}

\begin{lemma}\label{lem:uiticlose}
    For any \(P\in \mathscr{P}\), the parameters \(u_i\) and \(t_i\) \((i=1,\dots, 6)\) satisfies
    \begin{align*}
    0&=u_1-u_4+t_1-t_4-t_2+t_5-u_3+u_6-2t_3+2t_6
    \\ \text{and}\quad 0&=t_1-t_4+u_2-u_5+2t_2-2t_5+u_3-u_6+t_3-t_6.
    \end{align*}
\end{lemma}
\begin{proof}
    Since \(P\) is a closed polygon, we obtain the following equation:
    \begin{align*}
    0&=(u_1-u_4)g_1+(t_1-t_4)(g_1+g_2)+(u_2-u_5)g_2 +(t_2-t_5)(2g_2-g_1)\\
    &\mathrel{\phantom{=}} {}+(u_3-u_6)(g_2-g_1)+(t_3-t_6)(g_2-2g_1).  
    \end{align*}
    By separating the coefficients of the linearly independent \(g_1\) and \(g_2\), we obtain the statement of the lemma.
\end{proof}

\begin{lemma}\label{lem:bubt}
    For any \(P\in \mathscr{P}\), let \(b\) denote the number of boundary edges of \(P\), and let \(b_u=\sum_{i=1}^{6}{u_i}\) and \(b_t=\sum_{i=1}^{6}{t_i}\). If the interior angle of each boundary vertex of \(P\) is at least \(90^\circ\), then
    \[
    \partial(V(P))=6b_u+10b_t+12= 6b+4b_t+12.
    \]
    
\end{lemma}
\begin{proof}
    The \emph{deficit} of a vertex \(v\) of \(P\) is \(\operatorname{def}(v):=|\{ uv\in E(\Lambda_U):u\notin V(P)\}|.\) It is easy to see that 
    \[
    \partial(V(P))=\sum_{v\in V(P)}\operatorname{def}(v).
    \]
    It is immediate that the sum of the deficit over all boundary vertices of \(P\) is \(5b+12\). Now we calculate the sum of the deficit over all vertices not on the boundary. Let \(v\) be a vertex not on the boundary of \(P\) with a deficit of \(d>0\). This means \(v\) must be incident with \(d\) edges in \(\Lambda_U\) which are not in \(P\). Since \(P\) is an induced subgraph on the vertices contained within its boundary, each of these edges must cross the boundary of \(P\). These edges must cross the relative interior of a boundary edge of \(P\) since all edge crossings in \(\Lambda_U\) occur in the relative interior of edges.
    
    Furthermore, every edge of \(\Lambda_U\) that crosses a boundary edge of \(P\) must be incident to a vertex in \(P\). Suppose there was an edge of \(\Lambda_U\) that crossed the boundary of \(P\) and was not incident with a vertex of \(P\). Such an edge must cross two boundary edges of \(P\), implying the edge must be long. If the two boundary edges are incident with the same vertex, that vertex must have an interior angle of \(60^\circ\). If the two boundary edges are vertex disjoint they must be parallel long edges corresponding to opposite sides of \(P\). It is easy to see that such a \(P\) must have a boundary vertex with interior angle less than $90^\circ$.
    
    So to count the deficit of vertices not on the boundary, we only need to count the number of edges in \(\Lambda_U\) that cross boundary edges of \(P\). Observe that each short boundary edge is crossed once, and each long boundary edge is crossed \(5\) times. This implies the deficit of the vertices not on the boundary of \(P\) is \(b_u+5b_t\), which together with the deficit of \(5b+12\) coming from boundary vertices, implies the lemma.
\end{proof}

\subsection{Proof of Theorem \ref{the:P}}
\subsubsection{Base Cases}

We prove Theorem \ref{the:P} by induction on \(n\), starting with the following base cases.

\begin{claim}\label{cla:nleq30}
    For any \(P\in \mathscr{P}\) with \(n\) vertices and \(e\) edges, if \(n\in\{3,4,\dots, 30\}\) then \(e\leq e(n)\).
\end{claim}
\begin{proof}
    First we claim \(P\) must have a vertex of degree at most \(5\) in its boundary. If all boundary vertices had a degree of at least \(6\), then the boundary of \(P\) would have no degenerate sides. This implies \(P\) contains a \(12\)-gon with \(u_i=t_i=1\) \((i=1,\dots, 6)\), which implies \(n\geq 31\).
    
    We now prove the claim by finite induction. The case \(n\) equals \(3\) or \(4\) is trivial. If \(e(n)-e(n-1)\geq 5\) we are done, as by removing the boundary vertex of degree \(5\), we shift at least one supporting half-plane corresponding to this degenerate side inward. It is a simple case analysis to confirm the resulting graph is in \(\mathscr{P}\). One case is depicted in Figure~\ref{fig:degenremoveal}. Applying the inductive hypothesis after removing this vertex we obtain \(e\leq e(n-1)+5\leq e(n).\)

    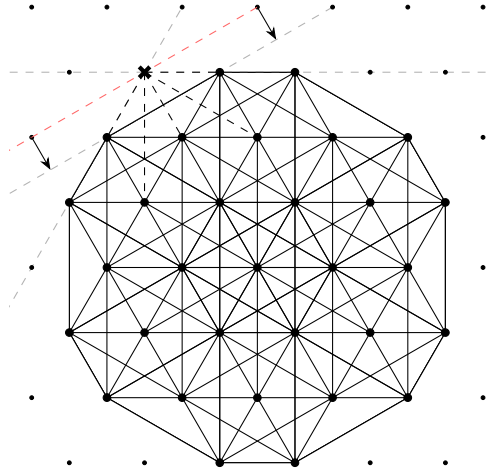
\begin{figure}
    \centering

\begin{tikzpicture}[>=Stealth]
\def\points{(e1),(e2),($(e1)+(e2)$),($2*(e2)-(e1)$),($(e2)-(e1)$), ($-2*(e1)+(e2)$),($-1*(e1)$),($-1*(e2)$),($-1*(e1)-(e2)$),($-2*(e2)+(e1)$),($-1*(e2)+(e1)$), ($2*(e1)-(e2)$)}
\def\bpoints{($2*(e1)-(e3)$), ($2*(e1)+(e2)$), ($2*(e2)+(e1)$), ($2*(e2)+(e3)$), ($2*(e3)+(e2)$), ($2*(e3)-(e1)$), ($-2*(e1)+(e3)$), ($-2*(e1)-(e2)$), ($-2*(e2)-(e1)$), ($-2*(e2)-(e3)$), ($-2*(e3)-(e2)$), ($-2*(e3)+(e1)$) }
\def\ipoints{(0,0), (e1), (e2), (e3), ($-1*(e1)$), ($-1*(e2)$), ($-1*(e3)$)}
\coordinate (e1) at (1,0);
\coordinate (e2) at (60:1);
\coordinate (e3) at (120:1);
\begin{scope}[on background layer]
\clip ($2*(e1)-(e3)$) --++ ($(e2)+(e3)$) --++ ($(e3)$) --++ ($-1*(e1)+(e3)$) --++ ($-1*(e1)$) --++ ($-1*(e1)-(e2)$) --++ ($-1*(e2)$) --++ ($-1*(e2)-(e3)$) --++ ($-1*(e3)$) --++ ($(e1)-(e3)$) --++ ($(e1)$) --++ ($(e1)+(e2)$) -- cycle;
\foreach \x in {-3,...,3} 
    {
    \foreach \y in {-3,...,3}
    {
    \coordinate (cur) at ($\x*(e1)+\y*(e2)$);
    \filldraw[draw=black] (cur) circle [radius=0.05];
    }
    }
\foreach \b in \bpoints {
    \foreach \p in \points{
        \draw \b -- +\p;
        }
    }
\end{scope}
\clip (-3.3,-3) rectangle (3.1,3.5);
\draw [dashed, gray!60!white] ($3*(e3)$) --($3*(e3)-5*(e2)$);
\draw [dashed, gray!60!white] ($3*(e3)$) --($3*(e3)+5*(e2)$);
\draw [dashed, gray!60!white] ($3*(e3)$) --($3*(e3)-5*(e1)$);
\draw [dashed, gray!60!white] ($3*(e3)$) --($3*(e3)+5*(e1)$);
\draw [dashed, gray!60!white] ($3*(e3)+2*(e1)+(e2)$) --($3*(e3)-5*(e2)-4*(e1)$);

\foreach \x in {-7,...,7} 
    {
    \foreach \y in {-7,...,7}
    {
    \coordinate (cur) at ($\x*(e1)+\y*(e2)$);
    \filldraw[draw=black] (cur) circle [radius=0.025];
    }
    }
\foreach \i in  \ipoints {
    \foreach \p in \points{
    \draw \i -- +\p;
    }
    }
\draw ($2*(e1)$) -- ($2*(e2)$) -- ($2*(e3)$) -- ($-2*(e1)$)-- ($-2*(e2)$) -- ($-2*(e3)$) -- cycle;

\draw ($2*(e1)-(e3)$) --++ ($(e2)+(e3)$) --++ ($(e3)$) --++ ($-1*(e1)+(e3)$) --++ ($-1*(e1)$) --++ ($-1*(e1)-(e2)$) --++ ($-1*(e2)$) --++ ($-1*(e2)-(e3)$) --++ ($-1*(e3)$) --++ ($(e1)-(e3)$) --++ ($(e1)$) --++ ($(e1)+(e2)$) -- cycle;

\draw ($(e1)+(e2)$) -- ($(e2)+(e3)$) -- ($(e3)-(e1)$) -- ($-1*(e1)-(e2)$)-- ($-1*(e2)-(e3)$) -- ($-1*(e3)+(e1)$) -- cycle;
\foreach \p in \bpoints 
    {
    \filldraw[draw=black] \p circle [radius=0.05];
    }
\draw [dashed, red!60!white] ($3*(e3)+(e1)+(e2)$) --($3*(e3)-5*(e2)-5*(e1)$);
\draw[ultra thick] ($3*(e3)+(45:0.1)$) --($3*(e3)+(225:0.1)$);
\draw[ultra thick] ($3*(e3)+(135:0.1)$) --($3*(e3)+(315:0.1)$);
\draw [dashed] ($3*(e3)$) --($3*(e3)+(e1)$);
\draw [dashed] ($3*(e3)$) --($3*(e3)+(e1)-(e3)$);
\draw [dashed] ($3*(e3)$) --($3*(e3)-(e3)$);
\draw [dashed] ($3*(e3)$) --($2*(e3)-(e2)$);
\draw [dashed] ($3*(e3)$) --($3*(e3)-(e2)$);
\draw [->] ($2*(e2)+2*(e3)$) -- +($-0.5*(e3)$);
\draw [->] ($-2*(e1)+2*(e3)$) -- +($-0.5*(e3)$);
\end{tikzpicture} 
    \caption{\(P\) after removal of a degenerate side remains in \(\mathscr{P}\)}
    \label{fig:degenremoveal}
    \end{figure}

    It is easy to check \(e(n)-e(n-1)\geq 5\) for all \(n\geq 3\) except when \(n\in \{3,4,5,6,8,9,11,13,15,20\}\). Since \(n\) equals \(3\) or \(4\) is done, we just have the remaining eight cases to check. \(n=5\) and \(6\) are easy. For the remaining cases we have \(e(n)-e(n-1)\geq 3\), so if we have a boundary vertex with degree at most \(3\) we are done by induction. So suppose each boundary vertex of \(P\) has a minimum degree of \(4\), this implies the interior angle of each boundary vertex is at least \(90^\circ\) allowing us to use Lemma \ref{lem:bubt}. For each \(n\in \{8,9,11,13,15,20\}\) we compute the side lengths of each \(P\in \mathscr{P}\) with \(n\) vertices using Lemma \ref{lem:nuiti} and Lemma \ref{lem:uiticlose}. For each \(P\) we compute the edge boundary using Lemma \ref{lem:bubt}, and confirm \(\partial(V(P))\geq 2\lceil\sqrt{96n-63}\rceil\) which implies \(e\leq e(n)\). Algorithm \ref{alg:basecase} states the pseudocode for this computation. The Python implementation can be found attached to this arXiv submission under the file name \texttt{Base\_case.py}.
\begin{table}
\centering\small
\begin{tabular}{c@{\;\;}c@{\;\;}c|c@{\;\;}c@{\;\;}c|c@{\;\;}c@{\;\;}c}
    $n$&$e(n)$&$e(n)-e(n-1)$&$n$&$e(n)$&$e(n)-e(n-1)$&$n$&$e(n)$&$e(n)-e(n-1)$ \\[0.5mm]
    \hline
    $3$&$3$&&$21$&$81$&$5$&$39$&$173$&$5$ \\
$4$&$6$&$3$&$22$&$86$&$5$&$40$&$178$&$5$ \\
$5$&$9$&$3$&$23$&$91$&$5$&$41$&$183$&$5$ \\
$6$&$13$&$4$&$24$&$96$&$5$&$42$&$189$&$6$ \\
$7$&$18$&$5$&$25$&$101$&$5$&$43$&$194$&$5$ \\
$8$&$21$&$3$&$26$&$106$&$5$&$44$&$199$&$5$ \\
$9$&$25$&$4$&$27$&$111$&$5$&$45$&$204$&$5$ \\
$10$&$30$&$5$&$28$&$116$&$5$&$46$&$210$&$6$ \\
$11$&$34$&$4$&$29$&$121$&$5$&$47$&$215$&$5$ \\
$12$&$39$&$5$&$30$&$126$&$5$&$48$&$220$&$5$ \\
$13$&$43$&$4$&$31$&$132$&$6$&$49$&$225$&$5$ \\
$14$&$48$&$5$&$32$&$137$&$5$&$50$&$231$&$6$ \\
$15$&$52$&$4$&$33$&$142$&$5$&$51$&$236$&$5$ \\
$16$&$57$&$5$&$34$&$147$&$5$&$52$&$241$&$5$ \\
$17$&$62$&$5$&$35$&$152$&$5$&$53$&$247$&$6$ \\
$18$&$67$&$5$&$36$&$157$&$5$&$54$&$252$&$5$ \\
$19$&$72$&$5$&$37$&$162$&$5$&$55$&$258$&$6$ \\
$20$&$76$&$4$&$38$&$168$&$6$&&& \\

\end{tabular}
\caption{The values of $e(n)$ and $e(n)-e(n-1)$ up to $n=55$}\label{tableBC}
\end{table}

\begin{algorithm}
\caption{The algorithm that computes the remaining base cases}\label{alg:basecase}
\begin{algorithmic}
\For{\(n \text{ in } [8,9,11,13,15,20]\)}
    \State \(\partial\text{\_bound}\gets2\lceil\sqrt{96n-63}\rceil\) \State \Comment{Want to show for all \(n\)-vertex \(P\in \mathscr{P}\), \(\partial(V(P))\geq \partial\text{\_bound}\).}
    \State \(b\text{\_bound}\gets \lceil(\partial\text{\_bound}-12)/6\rceil\)
    \State \Comment{By Lemma \ref{lem:bubt}, \(\partial(V(P))\geq \partial\text{\_bound}\) if \(b\geq b\text{\_bound}\).}

    \State \(\text{cases}\gets[(u_1,u_2,\dots,u_6,t_1,t_2,\dots,t_6):\text{satisfying \ref{con:1}, \ref{con:2}, \ref{con:3}, \ref{con:4}, \ref{con:5} and \ref{con:6}} ]\)
    \begin{enumerate}
        \item \(u_i,t_i\geq 0\), \((i=1,\dots, 6)\).\label{con:1}
        \item  \(\sum_{i=1}^6 u_i+t_i<b\text{\_bound}\). \label{con:2}
        \item \(0=u_1-u_4+t_1-t_4-t_2+t_5-u_3+u_6-2t_3+2t_6\).\label{con:3}
        \item \(0=t_1-t_4+u_2-u_5+2t_2-2t_5+u_3-u_6+t_3-t_6\).\label{con:4}
        \item \(n = (t_2+2t_3+t_4+u_3+u_4+1)(t_1+2t_2+t_3+u_2+u_3+1)\)
        
        \(\displaystyle\qquad-\binom{t_2+t_3+u_3+1}{2}-\binom{t_5+t_6+u_6+1}{2}-\sum_{i=1}^{6}{\binom{t_i+1}{2}}.\)\label{con:5}
        \item  \(\nexists i\in \{1,\dots, 6\}\) where \(0=u_i=t_i=u_{(i+1\mod{6})}\) or \(0=t_i=u_{(i+1\mod{6})}=t_{(i+1\mod{6})}\).\label{con:6}
    \end{enumerate}

    \Comment{Conditions \ref{con:1} and \ref{con:2} ensures \(P\) has non-negative sides and \(b<b\text{\_bound}\).}
    
    \Comment{Conditions \ref{con:3}, \ref{con:4} and \ref{con:5} ensures \(P\) satisfies Lemmas \ref{lem:nuiti} and \ref{lem:uiticlose}.}
    
    \Comment{Condition \ref{con:6} ensures interior angles on the boundary are at least \(90^\circ\).}

    \For{\((u_1,u_2,\dots,u_6,t_1,t_2,\dots,t_6)\) in cases}
        \State \(\partial(V(P))\gets \sum_{i=1}^6(6u_i+10t_i)+12\) \State \Comment{Computes \(\partial(V(P))\) according to Lemma \ref{lem:bubt}.}
        \If{\(\partial(V(P))< \partial\text{\_bound}\)}{ \Return \((u_1,u_2,\dots,u_6,t_1,t_2,\dots,t_6)\)} 
        \EndIf 
    \EndFor
\EndFor
\end{algorithmic}    
\end{algorithm}

\end{proof}

\subsubsection{Inductive step}
With the base cases done, we are set to prove Theorem \ref{the:P} by induction on \(n\). For the rest of this subsection, let \(P\in \mathscr{P}\) have \(n\geq 31\) vertices, \(e\) edges, and \(b\) boundary edges. Additionally, suppose \(P\) has the greatest number of edges out of all \(n\)-vertex members of \(\mathscr{P}\). The inductive hypothesis is that any \(P'\in \mathscr{P}\) with \(n'\) vertices, where \(3\leq n'<n\), has at most \(e(n')\) edges.

\begin{claim}\label{cla:deg5P}
    If \(P\) has a degenerate side, then \(e\leq 6n-\sqrt{96n-63}\). 
\end{claim}
\begin{proof}
    Suppose \(P\) has a degenerate side, then \(P\) has a boundary vertex with degree at most \(5\). The resulting graph after removal of this vertex remains in \(\mathscr{P}\), which we can then apply the inductive hypothesis to. Since \(n\geq 31\), \(6n-\sqrt{96n-63}-e(n-1)\geq 5\) which implies \(e\leq e(n-1)+5\leq 6n-\sqrt{96n-63}\).
\end{proof}

So we may assume \(P\) does not have a degenerate side. Moreover, we may assume there is no \(n\)-vertex \(P'\in \mathscr{P}\) with a degenerate side and \(e\) edges. Since \(P\) does not have a degenerate side, \(u_i,t_i\geq 1\) \((i=1,\dots, 6)\).

\begin{claim}\label{cla:maxmin}
    There exists an \(n^*\)-vertex \(P^*\in \mathscr{P}\) with \(e^*\) edges and parameters \(u_i^*, t_i^*\geq 1\) \((i=1,\dots,6)\) with the following properties:
    
    \begin{gather*}
    \max\{u_1^*,u_2^*,\dots, u_6^*\}-3\leq \min \{t_1^*,t_2^*,\dots, t_6^*\},\\
    n^*\leq n+\max\{u_1^*,u_2^*,\dots, u_6^*\},\\
    \text{and if }e^*\leq e(n^*) \text{ then }e\leq e(n).
    \end{gather*}

    \end{claim}
    
\begin{proof}
    If \(P\) has the property \(\max\{u_1,u_2,\dots, u_6\}-3\leq \min\{t_1,t_2,\dots, t_6\}\), we define \(P^*=P\). Otherwise, denote \(u_i=\max\{u_1,u_2,\dots, u_6\}\) and \(t_j=\min\{t_1,t_2,\dots, t_6\}\) and suppose  \(u_i-4\geq t_j\). As \(t_j\geq 1\) we must have \(u_i\geq 5\). Our approach is to first move the vertices of \(P\) around without decreasing the number of edges. We start by removing each vertex of \(P\) that makes up the side \(t_j\). There are \(t_j+1\) such vertices which when removed delete \(6(t_j+1)-1\) edges. We add these vertices back to \(P\) on the side \(u_i\). There are \(u_i-2\) vertices on this side to add and since we assumed \(u_i-4\geq t_j\) we can add each of the \(t_j+1\) vertices to the side \(u_i\), which creates \(6(t_j+1)-1\) edges. Figure~\ref{fig:ReArstart-2} depicts this process.

    \begin{figure}
    \centering
    \begin{subfigure}[t]{0.45\textwidth}
\begin{tikzpicture}[scale=0.45]
\def\points{(e1),(e2),($(e1)+(e2)$),($2*(e2)-(e1)$),($(e2)-(e1)$), ($-2*(e1)+(e2)$),($-1*(e1)$),($-1*(e2)$),($-1*(e1)-(e2)$),($-2*(e2)+(e1)$),($-1*(e2)+(e1)$), ($2*(e1)-(e2)$)}
\coordinate (e1) at (1,0);
\coordinate (e2) at (60:1);
\coordinate (e3) at (120:1);
\begin{scope}[on background layer]
\clip (-6.05,-5.5) rectangle (6.05,5);
\foreach \x in {-10,...,10} 
    {
    \foreach \y in {-10,...,10}
    {
    \coordinate (cur) at ($\x*(e1)+\y*(e2)$);
    \filldraw[draw=black] (cur) circle [radius=0.025];
    }
    }
\end{scope}
\begin{scope}
\clip ($5.1*(e1)$) -- ($3.1*(e2)+2.1*(e1)$) -- ($4.1*(e2)$) -- ($4.1*(e3)-1.1*(e1)$) -- ($3.1*(e3)-3*(e1)$) -- ($-5.1*(e1)+1.1*(e3)$) -- ($-5.1*(e1)-1.1*(e2)$) -- ($-4.1*(e2)-2.1*(e1)$) -- ($-5.1*(e2)$) -- ($-5.1*(e3)$) -- ($-4.1*(e3)+2.1*(e1)$) -- ($4.1*(e1)-2.1*(e3)$) -- cycle;
\filldraw[fill=gray!10!white]($5*(e1)$) -- ($3*(e2)+2*(e1)$) -- ($4*(e2)$) -- ($4*(e3)-(e1)$) -- ($3*(e3)-3*(e1)$) -- ($-5*(e1)+(e3)$) -- ($-5*(e1)-(e2)$) -- ($-4*(e2)-2*(e1)$) -- ($-5*(e2)$) -- ($-5*(e3)$) -- ($-4*(e3)+2*(e1)$) -- ($4*(e1)-2*(e3)$) -- cycle;
\foreach \x in {-7,...,7} 
    {
    \foreach \y in {-7,...,7}
    {
    \coordinate (cur) at ($\x*(e1)+\y*(e2)$);
    \filldraw[draw=black] (cur) circle [radius=0.05];
    }
    }
\end{scope}
\begin{scope}[every node/.style={scale=.45}]

\draw ($5*(e1)$) -- ($3*(e2)+2*(e1)$) node[midway, label={[label distance=0mm]180:$u_3=3$}] {};
\draw ($3*(e2)+2*(e1)$) -- ($4*(e2)$) node[midway, label={[label distance=1mm]180:$t_3=1$}] {};
\draw ($4*(e2)$) -- ($4*(e3)-(e1)$) node[midway, label={[label distance=0mm]270:$u_4=5$}] {};
\draw ($4*(e3)-(e1)$) -- ($3*(e3)-3*(e1)$) node[midway, label={[label distance=0mm]0:$t_4=1$}] {};
\draw ($3*(e3)-3*(e1)$) -- ($-5*(e1)+(e3)$) node[midway, label={[label distance=-1mm]300:$u_5=2$}] {};
\draw ($-5*(e1)+(e3)$) -- ($-5*(e1)-(e2)$) node[midway, label={[label distance=0mm]300:$t_5=1$}] {};
\draw ($-5*(e1)-(e2)$) -- ($-4*(e2)-2*(e1)$) node[midway, label={[label distance=0mm]0:$u_6=3$}] {};
\draw ($-4*(e2)-2*(e1)$) -- ($-5*(e2)$) node[midway, label={[label distance=0mm]0:$t_6=1$}] {};
\draw ($-5*(e2)$) -- ($-5*(e3)$) node[midway, label={[label distance=0mm]90:$u_1=5$}] {};
\draw ($-5*(e3)$) -- ($-4*(e3)+2*(e1)$) node[midway, label={[label distance=0mm]180:$t_1=1$}] {};
\draw ($-4*(e3)+2*(e1)$) -- ($4*(e1)-2*(e3)$) node[midway, label={[label distance=0mm]150:$u_2=2$}] {};
\draw ($4*(e1)-2*(e3)$) -- ($5*(e1)$) node[midway, label={[label distance=0mm]150:$t_2=1$}] {};
    
\end{scope}
\end{tikzpicture}
        \caption{Generic $P$}
    \end{subfigure}
    \hfill
    \begin{subfigure}[t]{0.45\textwidth}
\begin{tikzpicture}[scale=0.45]
\def\points{(e1),($(e1)+(e2)$),(e2),($2*(e2)-(e1)$),($(e2)-(e1)$), ($-2*(e1)+(e2)$),($-1*(e1)$),($-1*(e1)-(e2)$),($-1*(e2)$),($-2*(e2)+(e1)$),($-1*(e2)+(e1)$), ($2*(e1)-(e2)$)}
\coordinate (e1) at (1,0);
\coordinate (e2) at (60:1);
\coordinate (e3) at (120:1);
\begin{scope}[on background layer]
\clip (-6.05,-5.5) rectangle (6.05,5);
\foreach \x in {-10,...,10} 
    {
    \foreach \y in {-10,...,10}
    {
    \coordinate (cur) at ($\x*(e1)+\y*(e2)$);
    \filldraw[draw=black] (cur) circle [radius=0.025];
    }
    }
\end{scope}
\begin{scope}
\clip ($4.1*(e1)+1.1*(e2)$) -- ($3.1*(e2)+2.1*(e1)$) -- ($4.1*(e2)$) -- ($4.1*(e3)-1.1*(e1)$) -- ($3.1*(e3)-3*(e1)$) -- ($-5.1*(e1)+1.1*(e3)$) -- ($-5.1*(e1)-1.1*(e2)$) -- ($-4.1*(e2)-2.1*(e1)$) -- ($-5.1*(e2)$) -- ($-5.1*(e3)$) -- ($-4.1*(e3)+2.1*(e1)$) -- ($3.1*(e1)-3.1*(e3)$) -- cycle;
\filldraw[fill=gray!10!white]($4*(e1)+(e2)$) -- ($3*(e2)+2*(e1)$) -- ($4*(e2)$) -- ($4*(e3)-(e1)$) -- ($3*(e3)-3*(e1)$) -- ($-5*(e1)+(e3)$) -- ($-5*(e1)-(e2)$) -- ($-4*(e2)-2*(e1)$) -- ($-5*(e2)$) -- ($-5*(e3)$) -- ($-4*(e3)+2*(e1)$) -- ($3*(e1)-3*(e3)$) -- cycle;
\foreach \x in {-7,...,7} 
    {
    \foreach \y in {-7,...,7}
    {
    \coordinate (cur) at ($\x*(e1)+\y*(e2)$);
    \filldraw[draw=black] (cur) circle [radius=0.05];
    }
    }
\end{scope}
\begin{scope}[every node/.style={scale=.45}]

\draw ($4*(e1)+(e2)$) -- ($3*(e2)+2*(e1)$) node[midway, label={[label distance=0mm]210:$u_3=2$}] {};
\draw ($3*(e2)+2*(e1)$) -- ($4*(e2)$) node[midway, label={[label distance=1mm]180:$t_3=1$}] {};
\draw ($4*(e2)$) -- ($4*(e3)-(e1)$) node[midway, label={[label distance=0mm]270:$u_4=5$}] {};
\draw ($4*(e3)-(e1)$) -- ($3*(e3)-3*(e1)$) node[midway, label={[label distance=0mm]0:$t_4=1$}] {};
\draw ($3*(e3)-3*(e1)$) -- ($-5*(e1)+(e3)$) node[midway, label={[label distance=-1mm]300:$u_5=2$}] {};
\draw ($-5*(e1)+(e3)$) -- ($-5*(e1)-(e2)$) node[midway, label={[label distance=0mm]300:$t_5=1$}] {};
\draw ($-5*(e1)-(e2)$) -- ($-4*(e2)-2*(e1)$) node[midway, label={[label distance=0mm]0:$u_6=3$}] {};
\draw ($-4*(e2)-2*(e1)$) -- ($-5*(e2)$) node[midway, label={[label distance=0mm]0:$t_6=1$}] {};
\draw ($-5*(e2)$) -- ($-5*(e3)$) node[midway, label={[label distance=0mm]90:$u_1=5$}] {};
\draw ($-5*(e3)$) -- ($-4*(e3)+2*(e1)$) node[midway, label={[label distance=0mm]180:$t_1=1$}] {};
\draw ($-4*(e3)+2*(e1)$) -- ($3*(e1)-3*(e3)$) node[midway, label={[label distance=0mm]180:$u_2=1$}] {};
\draw ($3*(e1)-3*(e3)$) -- ($4*(e1)+(e2)$) node[midway, label={[label distance=2mm]150:$t_2=2$}] {};
    
\end{scope}

\def\atopoints{($(e2)-(e1)$), ($-2*(e1)+(e2)$),($-1*(e1)$),($-1*(e1)-(e2)$),($-1*(e2)$)}

\draw[thick] ($5*(e1)+(45:0.1)$) --($5*(e1)+(225:0.1)$);
\draw[thick] ($5*(e1)+(135:0.1)$) --($5*(e1)+(315:0.1)$);
\draw[thick] ($4*(e1)-2*(e3)+(135:0.1)$) --($4*(e1)-2*(e3)+(315:0.1)$);
\draw[thick] ($4*(e1)-2*(e3)+(45:0.1)$) --($4*(e1)-2*(e3)+(225:0.1)$);

\draw[dashed] ($5*(e1)$)--($4*(e1)-2*(e3)$);

\foreach \p in \atopoints
    {
    \draw[dashed] ($5*(e1)$)--+\p;
    \draw[dashed] ($4*(e1)-2*(e3)$)--+\p;
    }
    
\def\btopoints{($-1*(e1)-(e2)$),($-1*(e2)$),($-2*(e2)+(e1)$),($-1*(e2)+(e1)$), ($2*(e1)-(e2)$)}

\filldraw[draw=black] ($4*(e3)+(e2)$) circle [radius=0.05];
\filldraw[draw=black] ($3*(e3)+2*(e2)$) circle [radius=0.05];

\draw[] ($3*(e3)+2*(e2)$)--($4*(e3)+(e2)$);

\foreach \p in \btopoints
    {
    \draw[] ($3*(e3)+2*(e2)$)--+\p;
    \draw[] ($4*(e3)+(e2)$)--+\p;
    }
\end{tikzpicture}
        \caption{$P$ after}
    \end{subfigure}
    \caption{Removing vertices of $t_j=t_2$ and adding them to $u_i=u_4$}
    \label{fig:ReArstart-2}
\end{figure}
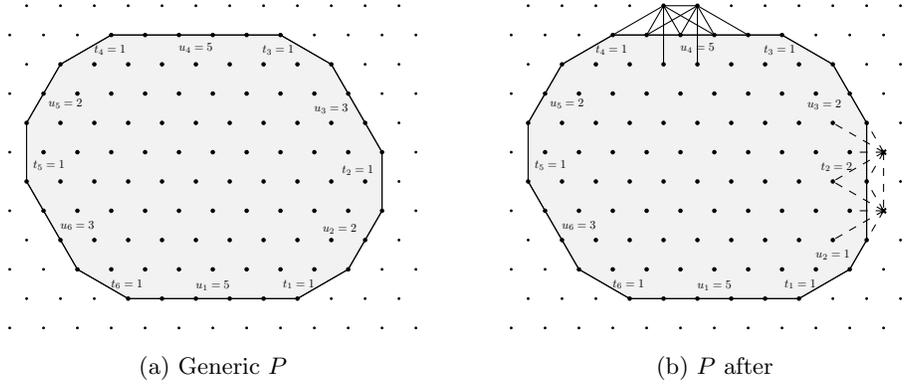

    By removing the vertices on side \(t_j\), we have increased \(t_j\) by \(1\) and decreased the two adjacent sides, \(u_j\) and \(u_{j+1}\), by \(1\). By adding vertices to the side \(u_i\), if the side was filled in completely, we would have decreased the side \(u_i\) by \(3\) and increased the two adjacent sides, \(t_{i-1}\) and \(t_i\), by \(1\). The only case we completely fill in the side \(u_i\), is when \(u_i\) and \(t_j\) are adjacent, and \(u_i=t_j-4\). Even though in most cases we will not completely fill in the side \(u_i\), when we complete the process described above, we update our parameters \(u_i,t_i\) \((i=1,\dots, 6)\) as if we had completely filled in the side \(u_i\). Namely 
    \[
    t_j=t_j+1, u_j=u_j-1, u_{j+1}=u_{j+1}-1,
    u_i=u_i-3, t_{i-1}=t_{i-1}+1,t_i=t_i+1. 
    \]
    
    Notice by updating these parameters we have strictly decreased \(\sum_{k=1}^{6}{u_k}\) and strictly increased \(\sum_{k=1}^{6}{t_k}\).
    So if we repeat this process, after a finite number of times, we will obtain \(\max \{u_1,u_2,\dots, u_6\}-3\leq \min \{t_1,t_2,\dots, t_6\}\).
    The only problem is that when we do this process once, we may have transformed \(P\) so that one of its sides is partially-filled, implying \(P\notin \mathscr{P}\).
    Figure~\ref{fig:Rear3} depicts the updated parameters where \(P\) has a side only partially-filled. 

    \begin{figure}
        \centering
\begin{tikzpicture}[scale=0.5]
\def\points{(e1),($(e1)+(e2)$),(e2),($2*(e2)-(e1)$),($(e2)-(e1)$), ($-2*(e1)+(e2)$),($-1*(e1)$),($-1*(e1)-(e2)$),($-1*(e2)$),($-2*(e2)+(e1)$),($-1*(e2)+(e1)$), ($2*(e1)-(e2)$)}
\coordinate (e1) at (1,0);
\coordinate (e2) at (60:1);
\coordinate (e3) at (120:1);
\filldraw[fill=gray!10!white]($4*(e1)+(e2)$) -- ($3*(e2)+2*(e1)$) -- ($3*(e2)+2*(e3)$) -- ($4*(e3)+(e2)$) -- ($3*(e3)-3*(e1)$) -- ($-5*(e1)+(e3)$) -- ($-5*(e1)-(e2)$) -- ($-4*(e2)-2*(e1)$) -- ($-5*(e2)$) -- ($-5*(e3)$) -- ($-4*(e3)+2*(e1)$) -- ($3*(e1)-3*(e3)$) -- cycle;
\begin{scope}[on background layer]
\clip (-6.05,-5.5) rectangle (6.05,5);
\foreach \x in {-10,...,10} 
    {
    \foreach \y in {-10,...,10}
    {
    \coordinate (cur) at ($\x*(e1)+\y*(e2)$);
    \filldraw[draw=black] (cur) circle [radius=0.025];
    }
    }
\end{scope}
\begin{scope}
\clip ($4.1*(e1)+1.1*(e2)$) -- ($3.1*(e2)+2.1*(e1)$) -- ($3*(e2)+2*(e3)$) -- ($4*(e3)+(e2)$) -- ($3.1*(e3)-3*(e1)$) -- ($-5.1*(e1)+1.1*(e3)$) -- ($-5.1*(e1)-1.1*(e2)$) -- ($-4.1*(e2)-2.1*(e1)$) -- ($-5.1*(e2)$) -- ($-5.1*(e3)$) -- ($-4.1*(e3)+2.1*(e1)$) -- ($3.1*(e1)-3.1*(e3)$) -- cycle;

\foreach \x in {-7,...,7} 
    {
    \foreach \y in {-7,...,7}
    {
    \coordinate (cur) at ($\x*(e1)+\y*(e2)$);
    \filldraw[draw=black] (cur) circle [radius=0.05];
    }
    }
\end{scope}
\begin{scope}[every node/.style={scale=.5}]

\draw ($4*(e1)+(e2)$) -- ($3*(e2)+2*(e1)$) node[midway, label={[label distance=0mm]210:$u_3=2$}] {};
\draw ($3*(e2)+2*(e1)$) -- ($3*(e2)+2*(e3)$) node[midway, label={[label distance=1mm]210:$t_3=2$}] {};
\draw ($3*(e2)+2*(e3)$) -- ($4*(e3)+(e2)$) node[midway, label={[label distance=0mm]270:$u_4=2$}] {};
\draw ($4*(e3)+(e2)$) -- ($3*(e3)-3*(e1)$) node[midway, label={[label distance=0mm]330:$t_4=2$}] {};
\draw ($3*(e3)-3*(e1)$) -- ($-5*(e1)+(e3)$) node[midway, label={[label distance=-1mm]300:$u_5=2$}] {};
\draw ($-5*(e1)+(e3)$) -- ($-5*(e1)-(e2)$) node[midway, label={[label distance=0mm]300:$t_5=1$}] {};
\draw ($-5*(e1)-(e2)$) -- ($-4*(e2)-2*(e1)$) node[midway, label={[label distance=0mm]0:$u_6=3$}] {};
\draw ($-4*(e2)-2*(e1)$) -- ($-5*(e2)$) node[midway, label={[label distance=0mm]0:$t_6=1$}] {};
\draw ($-5*(e2)$) -- ($-5*(e3)$) node[midway, label={[label distance=0mm]90:$u_1=5$}] {};
\draw ($-5*(e3)$) -- ($-4*(e3)+2*(e1)$) node[midway, label={[label distance=0mm]180:$t_1=1$}] {};
\draw ($-4*(e3)+2*(e1)$) -- ($3*(e1)-3*(e3)$) node[midway, label={[label distance=0mm]180:$u_2=1$}] {};
\draw ($3*(e1)-3*(e3)$) -- ($4*(e1)+(e2)$) node[midway, label={[label distance=2mm]150:$t_2=2$}] {};
    
\end{scope}

\def\btopoints{($-1*(e1)-(e2)$),($-1*(e2)$),($-2*(e2)+(e1)$),($-1*(e2)+(e1)$), ($2*(e1)-(e2)$)}

\filldraw[draw=black] ($4*(e3)+(e2)$) circle [radius=0.05];
\filldraw[draw=black] ($3*(e3)+2*(e2)$) circle [radius=0.05];
\filldraw[draw=black, thick, fill=white] ($2*(e3)+3*(e2)$) circle [radius=0.1];

\end{tikzpicture}
        \caption{Updated parameters of $P$ that is missing part of a side}
        \label{fig:Rear3}
    \end{figure}
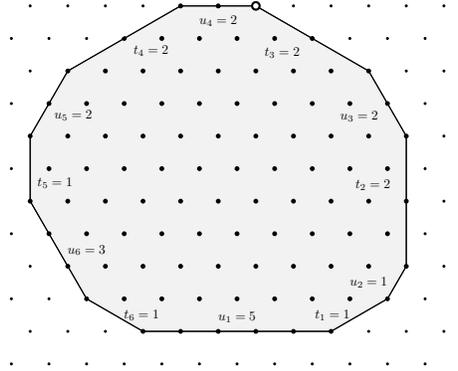

    This is surprisingly not a problem at the moment. When we repeat this process where \(P\) is a \(12\)-gon with one side only partially-filled, we identify \(t_{j}=\min\{t_1,t_2,\dots, t_6\}\) and \(u_i=\max\{u_1,u_2,\dots, u_6\}\). We remove the vertices on the side \(t_j\), then add these vertices to the partially-filled side to complete it before adding vertices to the side \(u_i\). By adding vertices to the partially-filled side, we add six edges per vertex so we still add at least \(6(t_j+1)-1\) edges back to \(P\). This is depicted in Figure~\ref{fig:ReAr4-5}. Note if \(t_j\) is adjacent to the partially-filled side, we shift the vertices of the partially-filled side so that \(t_j\) really has \(t_j\) many edges. 

    \begin{figure}
    \centering
    \begin{subfigure}[t]{0.45\textwidth}
\begin{tikzpicture}[scale=0.45]
\def\points{(e1),($(e1)+(e2)$),(e2),($2*(e2)-(e1)$),($(e2)-(e1)$), ($-2*(e1)+(e2)$),($-1*(e1)$),($-1*(e1)-(e2)$),($-1*(e2)$),($-2*(e2)+(e1)$),($-1*(e2)+(e1)$), ($2*(e1)-(e2)$)}
\coordinate (e1) at (1,0);
\coordinate (e2) at (60:1);
\coordinate (e3) at (120:1);
\filldraw[fill=gray!10!white]($4*(e1)+(e2)$) -- ($3*(e2)+2*(e1)$) -- ($3*(e2)+2*(e3)$) -- ($4*(e3)+(e2)$) -- ($3*(e3)-3*(e1)$) -- ($-5*(e1)+(e3)$) -- ($-5*(e1)-(e2)$) -- ($-4*(e2)-2*(e1)$) -- ($-5*(e2)$) -- ($-5*(e3)-(e1)$) -- ($3*(e1)-3*(e3)$) -- cycle;
\begin{scope}[on background layer]
\clip (-6.05,-5.5) rectangle (6.05,5);
\foreach \x in {-10,...,10} 
    {
    \foreach \y in {-10,...,10}
    {
    \coordinate (cur) at ($\x*(e1)+\y*(e2)$);
    \filldraw[draw=black] (cur) circle [radius=0.025];
    }
    }
\end{scope}
\begin{scope}
\clip ($4.1*(e1)+1.1*(e2)$) -- ($3.1*(e2)+2.1*(e1)$) -- ($3.1*(e2)+2.1*(e3)$) -- ($4.1*(e3)+1.1*(e2)$) -- ($3.1*(e3)-3*(e1)$) -- ($-5.1*(e1)+1.1*(e3)$) -- ($-5.1*(e1)-1.1*(e2)$) -- ($-4.1*(e2)-2.1*(e1)$) -- ($-5.1*(e2)$) -- ($-5.1*(e3)-1.1*(e1)$) -- ($3.1*(e1)-3.1*(e3)$) -- cycle;

\foreach \x in {-7,...,7} 
    {
    \foreach \y in {-7,...,7}
    {
    \coordinate (cur) at ($\x*(e1)+\y*(e2)$);
    \filldraw[draw=black] (cur) circle [radius=0.05];
    }
    }
\end{scope}
\begin{scope}[every node/.style={scale=.45}]

\draw ($4*(e1)+(e2)$) -- ($3*(e2)+2*(e1)$) node[midway, label={[label distance=0mm]210:$u_3=2$}] {};
\draw ($3*(e2)+2*(e1)$) -- ($3*(e2)+2*(e3)$) node[midway, label={[label distance=1mm]210:$t_3=2$}] {};
\draw ($3*(e2)+2*(e3)$) -- ($4*(e3)+(e2)$) node[midway, label={[label distance=0mm]240:$u_4=2$}] {};
\draw ($4*(e3)+(e2)$) -- ($3*(e3)-3*(e1)$) node[midway, label={[label distance=0mm]330:$t_4=2$}] {};
\draw ($3*(e3)-3*(e1)$) -- ($-5*(e1)+(e3)$) node[midway, label={[label distance=-1mm]300:$u_5=2$}] {};
\draw ($-5*(e1)+(e3)$) -- ($-5*(e1)-(e2)$) node[midway, label={[label distance=0mm]300:$t_5=1$}] {};
\draw ($-5*(e1)-(e2)$) -- ($-4*(e2)-2*(e1)$) node[midway, label={[label distance=0mm]0:$u_6=3$}] {};
\draw ($-4*(e2)-2*(e1)$) -- ($-5*(e2)$) node[midway, label={[label distance=0mm]0:$t_6=1$}] {};
\draw ($-5*(e2)$) -- ($-5*(e3)-(e1)$) node[midway, label={[label distance=0mm]90:$u_1=4$}] {};
\draw ($-5*(e3)-(e1)$) -- ($3*(e1)-3*(e3)$) node[midway, label={[label distance=3mm]160:$t_1=2$}] {};
\draw ($3*(e1)-3*(e3)$) node[label={[label distance=0mm]330:$u_2=0$}] {};
\draw ($3*(e1)-3*(e3)$) -- ($4*(e1)+(e2)$) node[midway, label={[label distance=2mm]150:$t_2=2$}] {};
    
\end{scope}
\filldraw[draw=black, thick, fill=white] ($2*(e3)+3*(e2)$) circle [radius=0.1];
    
\def\apoints{(e2),($2*(e2)-(e1)$),($(e2)-(e1)$), ($-2*(e1)+(e2)$),($-1*(e1)$)}

\draw[thick] ($2*(e1)-4*(e3)+(45:0.1)$) --($2*(e1)-4*(e3)+(225:0.1)$);
\draw[thick] ($2*(e1)-4*(e3)+(135:0.1)$) --($2*(e1)-4*(e3)+(315:0.1)$);
\draw[thick] ($-5*(e3)+(45:0.1)$) --($-5*(e3)+(225:0.1)$);
\draw[thick] ($-5*(e3)+(135:0.1)$) --($-5*(e3)+(315:0.1)$);

\foreach \p in \apoints
    {
    \draw[dashed] ($-5*(e3)$)--+\p;
    \draw[dashed] ($2*(e1)-4*(e3)$)--+\p;
    }
\draw[dashed] ($-5*(e3)$)--($2*(e1)-4*(e3)$);
\end{tikzpicture}
        \caption{Removal of $t_j=t_2$}
    \end{subfigure}
    \hfill
    \begin{subfigure}[t]{0.45\textwidth}
\begin{tikzpicture}[scale=0.45]
\def\points{(e1),($(e1)+(e2)$),(e2),($2*(e2)-(e1)$),($(e2)-(e1)$), ($-2*(e1)+(e2)$),($-1*(e1)$),($-1*(e1)-(e2)$),($-1*(e2)$),($-2*(e2)+(e1)$),($-1*(e2)+(e1)$), ($2*(e1)-(e2)$)}
\coordinate (e1) at (1,0);
\coordinate (e2) at (60:1);
\coordinate (e3) at (120:1);
\filldraw[fill=gray!10!white]($4*(e1)+(e2)$) -- ($3*(e2)+2*(e1)$) -- ($3*(e2)+2*(e3)$) -- ($4*(e3)+(e2)$) -- ($3*(e3)-3*(e1)$) -- ($-5*(e1)+(e3)$) -- ($-5*(e1)-(e2)$) -- ($-4*(e2)-2*(e1)$) -- ($-5*(e2)$) -- ($-5*(e3)-(e1)$) -- ($3*(e1)-3*(e3)$) -- cycle;
\begin{scope}[on background layer]
\clip (-6.05,-5.5) rectangle (6.05,5);
\foreach \x in {-10,...,10} 
    {
    \foreach \y in {-10,...,10}
    {
    \coordinate (cur) at ($\x*(e1)+\y*(e2)$);
    \filldraw[draw=black] (cur) circle [radius=0.025];
    }
    }
\end{scope}
\begin{scope}
\clip ($4.1*(e1)+1.1*(e2)$) -- ($3.1*(e2)+2.1*(e1)$) -- ($3.1*(e2)+2.1*(e3)$) -- ($4.1*(e3)+1.1*(e2)$) -- ($3.1*(e3)-3*(e1)$) -- ($-5.1*(e1)+1.1*(e3)$) -- ($-5.1*(e1)-1.1*(e2)$) -- ($-4.1*(e2)-2.1*(e1)$) -- ($-5.1*(e2)$) -- ($-5.1*(e3)-1.1*(e1)$) -- ($3.1*(e1)-3.1*(e3)$) -- cycle;

\foreach \x in {-7,...,7} 
    {
    \foreach \y in {-7,...,7}
    {
    \coordinate (cur) at ($\x*(e1)+\y*(e2)$);
    \filldraw[draw=black] (cur) circle [radius=0.05];
    }
    }
\end{scope}
\begin{scope}[every node/.style={scale=.45}]

\draw ($4*(e1)+(e2)$) -- ($3*(e2)+2*(e1)$) node[midway, label={[label distance=0mm]210:$u_3=2$}] {};
\draw ($3*(e2)+2*(e1)$) -- ($3*(e2)+2*(e3)$) node[midway, label={[label distance=1mm]210:$t_3=2$}] {};
\draw ($3*(e2)+2*(e3)$) -- ($4*(e3)+(e2)$) node[midway, label={[label distance=0mm]240:$u_4=2$}] {};
\draw ($4*(e3)+(e2)$) -- ($3*(e3)-3*(e1)$) node[midway, label={[label distance=0mm]330:$t_4=2$}] {};
\draw ($3*(e3)-3*(e1)$) -- ($-5*(e1)+(e3)$) node[midway, label={[label distance=-1mm]300:$u_5=2$}] {};
\draw ($-5*(e1)+(e3)$) -- ($-5*(e1)-(e2)$) node[midway, label={[label distance=0mm]300:$t_5=1$}] {};
\draw ($-5*(e1)-(e2)$) -- ($-4*(e2)-2*(e1)$) node[midway, label={[label distance=0mm]0:$u_6=3$}] {};
\draw ($-4*(e2)-2*(e1)$) -- ($-5*(e2)$) node[midway, label={[label distance=0mm]0:$t_6=1$}] {};
\draw ($-5*(e2)$) -- ($-5*(e3)-(e1)$) node[midway, label={[label distance=0mm]90:$u_1=4$}] {};
\draw ($-5*(e3)-(e1)$) -- ($3*(e1)-3*(e3)$) node[midway, label={[label distance=3mm]160:$t_1=2$}] {};
\draw ($3*(e1)-3*(e3)$) node[label={[label distance=0mm]330:$u_2=0$}] {};
\draw ($3*(e1)-3*(e3)$) -- ($4*(e1)+(e2)$) node[midway, label={[label distance=2mm]150:$t_2=2$}] {};
    
\end{scope}

\def\apoints{(e2),($2*(e2)-(e1)$),($(e2)-(e1)$), ($-2*(e1)+(e2)$),($-1*(e1)$)}

\draw[thick] ($2*(e1)-4*(e3)+(45:0.1)$) --($2*(e1)-4*(e3)+(225:0.1)$);
\draw[thick] ($2*(e1)-4*(e3)+(135:0.1)$) --($2*(e1)-4*(e3)+(315:0.1)$);
\draw[thick] ($-5*(e3)+(45:0.1)$) --($-5*(e3)+(225:0.1)$);
\draw[thick] ($-5*(e3)+(135:0.1)$) --($-5*(e3)+(315:0.1)$);

\foreach \p in \apoints
    {
    \draw[dashed] ($-5*(e3)$)--+\p;
    \draw[dashed] ($2*(e1)-4*(e3)$)--+\p;
    }
\draw[dashed] ($-5*(e3)$)--($2*(e1)-4*(e3)$);
\def\bpoints{($-1*(e1)-(e2)$),($-1*(e2)$),($-2*(e2)+(e1)$),($-1*(e2)+(e1)$)}

\foreach \p in \bpoints
    {
    \draw[] ($2*(e3)+3*(e2)$)--+\p;
    }
\filldraw[draw=black] ($-4*(e2)-2*(e3)$) circle [radius=0.05];
\def\cpoints{($(e1)+(e2)$),(e2),($2*(e2)-(e1)$),($(e2)-(e1)$), ($-2*(e1)+(e2)$)}
\foreach \p in \cpoints
    {
    \draw[] ($-4*(e2)-2*(e3)$) --+\p;
    }

\end{tikzpicture}
        \caption{completing the side $u_4$ then adding to $u_i=u_1$}
        \label{fig:ReAr4-5b}
    \end{subfigure}
    \caption{Process when there is a partially-filled side}
    \label{fig:ReAr4-5}
    \end{figure}
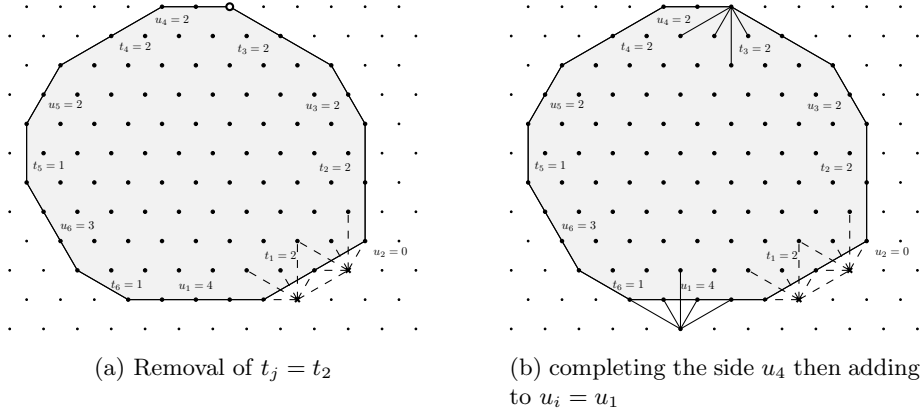

    There is one more issue this process might run into. In the process of doing this rearrangement we may create a degenerate side. Since each of the sides of \(P\) corresponding to the positive parameters \(t_1,t_2,\dots,t_6\) can only increase, this only happens for sides corresponding to the parameters \(u_1,u_2,\dots, u_6\).

    Since \(5\leq u_i=\max\{u_1,u_2,\dots, u_6\}\) and \(u_i\) is decreased by at most \(4\), \(u_i\) will never be \(0\). So we can only form a degenerate side by removing the vertices of \(t_j=\min\{t_1,t_2,\dots, t_6\}\) where \(u_j=1\) or \(u_{j+1}=1\) or both. In any case, we add the vertices of \(t_j\) to \(u_i\). At this stage, we have transformed \(P\) without decreasing its edges, into a \(12\)-gon with at least one degenerate side, and at most one side that is partially-filled. Figure~\ref{fig:ReAr4-5b} depicts an example of this.
    
    We claim that at this point in the process, \(P\) must have one side that is only partially-filled. If there is no partially-filled side, then \(P\in \mathscr{P}\). This is a contradiction as we supposed there was no \(n\)-vertex member of \(\mathscr{P}\) with a degenerate side and \(e\) edges.
    
    So there is a partially-filled side of \(P\). We remove the degree at most $5$ vertex corresponding to a degenerate side, and add it to the partially-filled side which creates $6$ edges. Notice that the edges in \(P\) have strictly increased, and since \(P\) can have at most \(6n\) edges, this case happens only a finite number of times. This is depicted in Figure~\ref{fig:ReAr6-7}.

    \begin{figure}
    \centering
    \begin{subfigure}[t]{0.45\textwidth}
\begin{tikzpicture}[scale=0.45]
\def\points{(e1),($(e1)+(e2)$),(e2),($2*(e2)-(e1)$),($(e2)-(e1)$), ($-2*(e1)+(e2)$),($-1*(e1)$),($-1*(e1)-(e2)$),($-1*(e2)$),($-2*(e2)+(e1)$),($-1*(e2)+(e1)$), ($2*(e1)-(e2)$)}
\coordinate (e1) at (1,0);
\coordinate (e2) at (60:1);
\coordinate (e3) at (120:1);
\filldraw[fill=gray!10!white]($4*(e1)+(e2)$) -- ($3*(e2)+2*(e1)$) -- ($3*(e2)+2*(e3)$) -- ($4*(e3)+(e2)$) -- ($3*(e3)-3*(e1)$) -- ($-5*(e1)+(e3)$) -- ($-5*(e1)-(e2)$) -- ($-4*(e2)-2*(e1)$) -- ($-5*(e2)+(e1)-(e3)$) -- ($-5*(e3)-2*(e1)-(e2)$) -- ($3*(e1)-3*(e3)$) -- cycle;
\begin{scope}[on background layer]
\clip (-6.05,-5.5) rectangle (6.05,5);
\foreach \x in {-10,...,10} 
    {
    \foreach \y in {-10,...,10}
    {
    \coordinate (cur) at ($\x*(e1)+\y*(e2)$);
    \filldraw[draw=black] (cur) circle [radius=0.025];
    }
    }
\end{scope}
\begin{scope}
\clip ($4.1*(e1)+1.1*(e2)$) -- ($3.1*(e2)+2.1*(e1)$) -- ($3.1*(e2)+2.1*(e3)$) -- ($4.1*(e3)+1.1*(e2)$) -- ($3.1*(e3)-3*(e1)$) -- ($-5.1*(e1)+1.1*(e3)$) -- ($-5.1*(e1)-1.1*(e2)$) -- ($-4.1*(e2)-2.1*(e1)$) -- ($-5.1*(e2)+1.1*(e1)-1.1*(e3)$) -- ($-5.1*(e3)-2.1*(e1)-1.1*(e2)$) -- ($3.1*(e1)-3.1*(e3)$) -- cycle;

\foreach \x in {-7,...,7} 
    {
    \foreach \y in {-7,...,7}
    {
    \coordinate (cur) at ($\x*(e1)+\y*(e2)$);
    \filldraw[draw=black] (cur) circle [radius=0.05];
    }
    }
\end{scope}
\begin{scope}[every node/.style={scale=.45}]

\draw ($4*(e1)+(e2)$) -- ($3*(e2)+2*(e1)$) node[midway, label={[label distance=0mm]210:$u_3=2$}] {};
\draw ($3*(e2)+2*(e1)$) -- ($3*(e2)+2*(e3)$) node[midway, label={[label distance=1mm]210:$t_3=2$}] {};
\draw ($3*(e2)+2*(e3)$) -- ($4*(e3)+(e2)$) node[midway, label={[label distance=0mm]270:$u_4=2$}] {};
\draw ($4*(e3)+(e2)$) -- ($3*(e3)-3*(e1)$) node[midway, label={[label distance=0mm]330:$t_4=2$}] {};
\draw ($3*(e3)-3*(e1)$) -- ($-5*(e1)+(e3)$) node[midway, label={[label distance=-1mm]300:$u_5=2$}] {};
\draw ($-5*(e1)+(e3)$) -- ($-5*(e1)-(e2)$) node[midway, label={[label distance=0mm]300:$t_5=1$}] {};
\draw ($-5*(e1)-(e2)$) -- ($-4*(e2)-2*(e1)$) node[midway, label={[label distance=0mm]0:$u_6=3$}] {};
\draw ($-4*(e2)-2*(e1)$) -- ($-5*(e2)+(e1)-(e3)$) node[midway, label={[label distance=0mm]30:$t_6=2$}] {};
\draw ($-5*(e2)+(e1)-(e3)$) -- ($-5*(e3)-2*(e1)-(e2)$) node[midway, label={[label distance=0mm]90:$u_1=1$}] {};
\draw ($-5*(e3)-2*(e1)-(e2)$) -- ($3*(e1)-3*(e3)$) node[midway, label={[label distance=1mm]180:$t_1=3$}] {};
\draw ($3*(e1)-3*(e3)$) node[label={[label distance=0mm]330:$u_2=0$}] {};
\draw ($3*(e1)-3*(e3)$) -- ($4*(e1)+(e2)$) node[midway, label={[label distance=2mm]150:$t_2=2$}] {};
    
\end{scope}
\filldraw[draw=black, thick, fill=white] ($-3*(e3)-3*(e2)$) circle [radius=0.1];

\end{tikzpicture}
        \caption{\(P\) with one degenerate side and one partially-filled side}
    \end{subfigure}
    \hfill
    \begin{subfigure}[t]{0.45\textwidth}
\begin{tikzpicture}[scale=0.45]
\def\points{(e1),($(e1)+(e2)$),(e2),($2*(e2)-(e1)$),($(e2)-(e1)$), ($-2*(e1)+(e2)$),($-1*(e1)$),($-1*(e1)-(e2)$),($-1*(e2)$),($-2*(e2)+(e1)$),($-1*(e2)+(e1)$), ($2*(e1)-(e2)$)}
\coordinate (e1) at (1,0);
\coordinate (e2) at (60:1);
\coordinate (e3) at (120:1);
\filldraw[fill=gray!10!white]($4*(e1)+(e2)$) -- ($3*(e2)+2*(e1)$) -- ($3*(e2)+2*(e3)$) -- ($4*(e3)+(e2)$) -- ($3*(e3)-3*(e1)$) -- ($-5*(e1)+(e3)$) -- ($-5*(e1)-(e2)$) -- ($-4*(e2)-2*(e1)$) -- ($-5*(e2)+(e1)-(e3)$) -- ($-5*(e3)-2*(e1)-(e2)$) -- ($-4*(e3)+1*(e1)$) -- ($4*(e1)-1*(e3)$) -- cycle;
\begin{scope}[on background layer]
\clip (-6.05,-5.5) rectangle (6.05,5);
\foreach \x in {-10,...,10} 
    {
    \foreach \y in {-10,...,10}
    {
    \coordinate (cur) at ($\x*(e1)+\y*(e2)$);
    \filldraw[draw=black] (cur) circle [radius=0.025];
    }
    }
\end{scope}
\begin{scope}
\clip ($4.1*(e1)+1.1*(e2)$) -- ($3.1*(e2)+2.1*(e1)$) -- ($3.1*(e2)+2.1*(e3)$) -- ($4.1*(e3)+1.1*(e2)$) -- ($3.1*(e3)-3*(e1)$) -- ($-5.1*(e1)+1.1*(e3)$) -- ($-5.1*(e1)-1.1*(e2)$) -- ($-4.1*(e2)-2.1*(e1)$) -- ($-5.1*(e2)+1.1*(e1)-1.1*(e3)$) -- ($-5.1*(e3)-2.1*(e1)-1.1*(e2)$) -- ($-4.1*(e3)+1.1*(e1)$) -- ($4.1*(e1)-1.1*(e3)$) -- cycle;

\foreach \x in {-7,...,7} 
    {
    \foreach \y in {-7,...,7}
    {
    \coordinate (cur) at ($\x*(e1)+\y*(e2)$);
    \filldraw[draw=black] (cur) circle [radius=0.05];
    }
    }
\end{scope}
\begin{scope}[every node/.style={scale=.45}]

\draw ($4*(e1)+(e2)$) -- ($3*(e2)+2*(e1)$) node[midway, label={[label distance=0mm]210:$u_3=2$}] {};
\draw ($3*(e2)+2*(e1)$) -- ($3*(e2)+2*(e3)$) node[midway, label={[label distance=1mm]210:$t_3=2$}] {};
\draw ($3*(e2)+2*(e3)$) -- ($4*(e3)+(e2)$) node[midway, label={[label distance=0mm]270:$u_4=2$}] {};
\draw ($4*(e3)+(e2)$) -- ($3*(e3)-3*(e1)$) node[midway, label={[label distance=0mm]330:$t_4=2$}] {};
\draw ($3*(e3)-3*(e1)$) -- ($-5*(e1)+(e3)$) node[midway, label={[label distance=-1mm]300:$u_5=2$}] {};
\draw ($-5*(e1)+(e3)$) -- ($-5*(e1)-(e2)$) node[midway, label={[label distance=0mm]300:$t_5=1$}] {};
\draw ($-5*(e1)-(e2)$) -- ($-4*(e2)-2*(e1)$) node[midway, label={[label distance=0mm]0:$u_6=3$}] {};
\draw ($-4*(e2)-2*(e1)$) -- ($-5*(e2)+(e1)-(e3)$) node[midway, label={[label distance=0mm]30:$t_6=2$}] {};
\draw ($-5*(e2)+(e1)-(e3)$) -- ($-5*(e3)-2*(e1)-(e2)$) node[midway, label={[label distance=0mm]270:$u_1=1$}] {};
\draw ($-5*(e3)-2*(e1)-(e2)$) -- ($-4*(e3)+1*(e1)$) node[midway, label={[label distance=1mm]150:$t_1=2$}] {};
\draw ($-4*(e3)+1*(e1)$) -- ($4*(e1)-1*(e3)$) node[midway, label={[label distance=3mm]180:$u_2=3$}] {};
\draw ($4*(e1)-1*(e3)$) -- ($4*(e1)+(e2)$) node[midway, label={[label distance=2mm]150:$t_2=1$}] {};
    
\end{scope}

\def\apoints{(e2),($2*(e2)-(e1)$),($(e2)-(e1)$), ($-2*(e1)+(e2)$)}
\foreach \p in \apoints
    {
    \draw[] ($-3*(e3)-3*(e2)$)--+\p;
    }
\draw[thick] ($3*(e1)-3*(e3)+(45:0.1)$) --($3*(e1)-3*(e3)+(225:0.1)$);
\draw[thick] ($3*(e1)-3*(e3)+(135:0.1)$) --($3*(e1)-3*(e3)+(315:0.1)$);

\def\bpoints{($2*(e2)-(e1)$),($(e2)-(e1)$), ($-2*(e1)+(e2)$),($-1*(e1)$),($-1*(e1)-(e2)$)}
\foreach \p in \bpoints
    {
    \draw[dashed] ($3*(e1)-3*(e3)$)--+\p;
    }

\end{tikzpicture}
        \caption{Removal of the degenerate side and adding it to the partially-filled side}
    \end{subfigure}
    \caption{Depiction of process when a degenerate side is created}
    \label{fig:ReAr6-7}
    \end{figure}
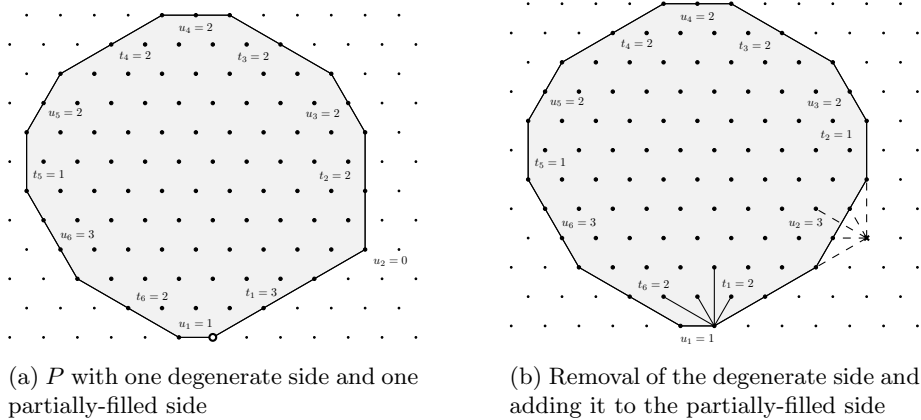
    
    After removing a degenerate side corresponding to the parameter \(u_k\), \(u_k\) is increased by \(3\), and its adjacent sides \(t_{k-1}\) and \(t_k\) have decreased by \(1\). This may form new degenerate sides, but then we repeat this process: Consider a degenerate side. Just as before there must be some other side that is half filled, so remove the vertex of degree at most \(5\) corresponding to the degenerate side, and add it to the partially-filled side strictly increasing the number of edges. The only way this process stops is when \(\max\{u_1,u_2,\dots, u_6\}-3\leq \min\{t_1,t_2,\dots, t_6\}\).
    
    Let \(P'\) denote the end result of this process, with \(n\) vertices, \(e'\geq e\) edges, and parameters \(u_i',t_i'\geq 1\) \((i=1,\dots, 6)\) satisfying
    \[
    \max\{u_1',u_2',\dots, u_6'\}-3\leq \min\{t_1',t_2',\dots, t_6'\}.
    \]
    
    Since \(P'\) may have an incomplete side, it may not be in \(\mathscr{P}\). If \(P'\in \mathscr{P}\), we let \(P^*=P'\) with $n^*=n$ vertices and $e^*=e'$ edges. Otherwise, let \(P^*\in \mathscr{P}\) be \(P'\) with the incomplete side filled in, with \(n^* > n\) vertices and \(e^*\) edges. We add at most \(\max \{u_1^*,u_2^*,\dots, u_6^*\}\) vertices to fill in this side. If we can show that \(P^*\) satisfies Theorem \ref{the:P} then we are done as \(e^*=e'+6(n^*-n)\) and so if \(e^*\leq e(n^*)\) then 
    \[
    e\leq e'\leq e^*-6(n^*-n)\leq e(n^*)-6(n^*-n)\leq e(n). \qedhere
    \]
    
    \end{proof}

Continuing with the proof of Theorem~\ref{the:P}, let \(P^*\in \mathscr{P}\) have \(n^*\) vertices, \(e^*\) edges, and \(b^*\) boundary edges, be the \(12\)-gon associated to \(P\) given by Claim \ref{cla:maxmin}. By Claim \ref{cla:maxmin}, to complete the inductive step for \(P\), it suffices to prove \(e^*\leq e(n^*)\). Note that, also by Claim \ref{cla:maxmin}, the inductive hypothesis applies for all
\[
n'<n^*-\max \{u_1^*,u_2^*,\dots, u_6^*\}\leq n.
\]

\begin{definition}\label{def:dudt}
    Let \(k=\max\{u^*_1,u^*_2,\dots, u^*_6\}\). Let \(\mu^*_i=k-u^*_i\) and \(\tau^*_i=t^*_i-(k-3)\) \((i=1,\dots,6)\), and
    \[
    d_u^*=\sum_{i=1}^{6}{\mu^*_i}\text{ and } d_t^*=\sum_{i=1}^{6}{\tau^*_i}.
    \]
\end{definition}
By Claim \ref{cla:maxmin} we have that \(\mu^*_i,\tau^*_i\geq 0\) \((i=1,\dots, 6)\), hence \(d_u^*\) and \(d_t^*\) are non-negative. An intuitive way to view these parameters is that \(\mu^*_i\) and \(\tau^*_i\) measure how far away \(u^*_i\) and \(t^*_i\) are from their respective maximum and minimum. Let \(b^*_u=\sum_{i=1}^{6}{u_i^*}\) and \(b^*_t=\sum_{i=1}^{6}{t_i^*}\).

\begin{claim}\label{cla:bubt}
    \(b_t^*=b_u^*+d_u^*+d_t^*-18.\)
\end{claim}
\begin{proof}
    \[
    b_t^*=\sum_{i=1}^{6}{t_i^*}=6k-18+d_t^*=\sum_{i=1}^{6}{u_i^*}+d_u^*-18+d_t^*=b_u^*+d_u^*+d_t^*-18. \qedhere 
    \]
\end{proof}

\begin{claim}\label{cla:e*}
    \(
    e^*=6n^*-4b^*-d_u^*-d_t^*+12
    \)
\end{claim}
\begin{proof}
    By Lemma \ref{lem:bubt} and Claim \ref{cla:bubt} we have 
    \begin{align*}
    \partial(V(P^*))&=6b^*+4b_t^*+12
    =6b^*+2(b^*+d_u^*+d_t^*-18)+12
    \\&=8b^*+2(d_u^*+d_t^*-12).
    \end{align*}
    So
    \[
    e^*= 6n^*-\frac{\partial(V(P^*))}{2}=6n^*-4b^*-d_u^*-d_t^*+12.\qedhere
    \]
\end{proof}

\begin{claim}\label{cla:bbig}
    Suppose \(b^*\geq \frac{\sqrt{96n^*-63}}{4}-\frac{3}{2}\), if \(d_u^*+d_t^*\geq 18\) then \(e^*\leq 6n^*-\sqrt{96n^*-63}\).
\end{claim}
\begin{proof}
    Using Claim \ref{cla:e*} we obtain
    \begin{align*}
    e^*&=6n^*-4b^*-d_u^*-d_t^*+12\\
    &\leq 6n^*-4\left(\frac{\sqrt{96n^*-63}}{4}-\frac{3}{2}\right)-18+12
    = 6n^*-\sqrt{96n^*-63}. \qedhere
    \end{align*}
\end{proof}

Now we handle the inductive step if the boundary is sufficiently small. Let \(r^*\) be the number of edges in~\(P^*\) incident to a boundary vertex of \(P^*\), and let \(d_i^*\) be the number of boundary vertices of \(P^*\) having degree~\(i\). The minimum interior angle between two edges incident with the same vertex is \(30^\circ\), which together with the angle sum formula for polygons applied to the boundary of \(P^*\), implies that \(30^\circ r^*=30^\circ\sum_{i}d_i^*(i-1)\leq 180^\circ(b^*-2) \) giving us
\begin{equation}\label{eq:r}
    r^*\leq 6b^*-12. 
\end{equation}

\begin{claim}\label{cla:bsmall}
    If \(b^*\leq \frac{\sqrt{96n^*-63}}{4}-\frac{3}{2}\), then \(e^*\leq 6n^*-\sqrt{96n^*-63}.\)
\end{claim}

\begin{proof}
    Remove each boundary vertex from \(P^*\) to form an \(n'\)-vertex graph \(P'\) with $e'$ edges. Since each \(u_i^*,t_i^*\geq 1\), \(P^*\) contains the \(12\)-gon with \(u_i=t_i=1\) \((i=1,\dots,6)\), implying that \(P'\) has at least \(8\) vertices. Removing the boundary of \(P^*\) shifts each half-plane defining \(P^*\) inwards, which gives \(P'\in \mathcal{P^*}\). Since \(n'=n^*-b^*<n^*-\max \{u_1^*,u_2^*,\dots, u_6^*\}\), we may apply the inductive hypothesis to \(P'\).
    There are two cases: either \(P'\) has parameters \(u_i=k\) and \(t_i=k-1\) \((i=1,\dots, 6)\) for some \(k\in \mathbb{N}\) or \(P'\) does not.
    
    In the case where \(P'\) has parameters \(u_i=k\) and \(t_i=k-1\) \((i=1,\dots, 6)\) for some \(k\in \mathbb{N}\), we see that \(P^*\) must be the graph where all \(u_i^*=k-1\) and \(t_i^*=k\) (removing the boundary of \(P^*\) decreases each \(t_i^*\) by one and increases each \(u_i^*\) by one) which implies \( b_t^*=b_u^*+6\), which by Lemma \ref{lem:bubt} implies
    \[
    \partial(V(P^*))=6b^*+4b_t^*+12=6b^*+2(b_t^*+b_u^*+6)+12=8b^*+24.
    \]
    Using Lemma~\ref{lem:nuiti} we obtain \(n^*=24k^2-12k+1\). Thus \(k=\frac{1}{4}+\frac{\sqrt{2n+1}}{4\sqrt{3}}\), hence \(b^*=12k-6=\sqrt{6n+3}-3\),
    which gives us \(\partial(V(P^*))=8\sqrt{6n+3}\) implying by \eqref{eq:eindef} that \(e^*=6n-4\sqrt{6n+3}\leq e(n^*)\).
    
    In the second case we obtain by \eqref{eq:r}:
    \begin{align*}
    e^*=e'+r^*&\leq 6(n^*-b^*)-\sqrt{96(n^*-b^*)-63}+6b^*-12
    \\&=6n^*-\sqrt{96(n^*-b^*)-63}-12.
    \end{align*}
    So to complete this case it remains to be shown that $6n^*-\sqrt{96(n^*-b^*)-63}-12\leq 6n^*-\sqrt{96n^*-63}$. However, this inequality is algebraically equivalent to the given $b^*\leq \frac{\sqrt{96n^*-63}}{4}-\frac{3}{2}$.

\end{proof}

\subsubsection{Remaining cases in the inductive step}

We have proven the inductive step of Theorem \ref{the:P} for all cases except when \(\frac{\sqrt{96n^*-63}}{4}-\frac{3}{2}<b^*\) and \(d_u^*+d_t^*<18\). For each \(k\in \mathbb{N}\) there is only a finite number of cases the twelve parameters \(\mu^*_i\) and \(\tau^*_i\) \((i=1,\dots,6)\) can take as their sum is less than \(18\). First we substitute the parameters \(k, \mu^*_i, \tau^*_i\) into Lemma \ref{lem:uiticlose} to obtain the following condition on the \(\mu^*_i\) and \(\tau^*_i\) \((i=1,\dots, 6)\) in order for them to be the parameters of a \(12\)-gon.

\begin{align*}
    0&=-\mu_1^*+\mu_4^*+\tau_1^*-\tau_4^*-\tau_2^*+\tau_5^*+\mu_3^*-\mu_6^*-2\tau_3^*+2\tau_6^*
    \\ 0&=\tau_1^*-\tau_4^*-\mu_2^*+\mu_5^*+2\tau_2^*-2\tau_5^*-\mu_3^*+\mu_6^*+\tau_3^*-\tau_6^*
\end{align*}

We can also make this substitution into Lemma \ref{lem:nuiti} to obtain the following formula for \(n^*\) in terms of \(k,\mu^*_i\) and \(\tau^*_i\) \((i=1,\dots, 6)\).

\begin{align*}
    n^*&=(6k-11+\tau_2^*+2\tau_3^*+\tau_4^*-\mu_3^*-\mu_4^*)(6k-11+\tau_1^*+2\tau_2^*+\tau_3^*+\mu_2^*+\mu_3^*)
    \\& \\&\mathrel{\phantom{=}} {} -\binom{3k-5+\tau_2^*+\tau_3^*+\mu_3^*}{2}-\binom{3k-5+\tau_5^*+\tau_6^*+\mu_6^*}{2}
    \\& \\&\mathrel{\phantom{=}} {}-\sum_{i=1}^{6}{\binom{k-2+\tau_i^*}{2}}.
\end{align*}

This simplifies to 
\begin{equation}\label{eq:LQ}
\begin{aligned}
n^*&=24k^2+L(\tau^*_1,\tau^*_2,\tau^*_3,\tau^*_4,\tau^*_5,\tau^*_6,\mu^*_2,\mu^*_3,\mu^*_4,\mu^*_6)k
\\&\mathrel{\phantom{=}} {}+Q(\tau^*_1,\tau^*_2,\tau^*_3,\tau^*_4,\tau^*_5,\tau^*_6,\mu^*_2,\mu^*_3,\mu^*_4,\mu^*_6), 
\end{aligned}
\end{equation}
where \(L\) is a linear function and \(Q\) is a quadratic function. This gives
\begin{equation}\label{eq:k}
k=\frac{-L+ \sqrt{L^2-96(Q-n^*)}}{48}.
\end{equation}

We can also obtain \(b^*\) in terms of \(k\) \(\mu_i^*\) and \(\tau_i^*\) \((i=1,\dots, 6)\). 

\begin{equation}\label{eq:b^*}
b^*=b_u^*+b_t^*=\sum_{i=1}^{6}{u_i^*}+\sum_{i=1}^{6}{t_i^*}=6k-\sum_{i=1}^{6}{\mu_i^*}+6k-18+\sum_{i=1}^{6}{\tau_i^*}=12k-18-d_u^*+d_t^*.
\end{equation}

Substituting \eqref{eq:k} and \eqref{eq:b^*} into Claim \ref{cla:e*} we obtain:
\begin{align*}
e^*&=6n^*-4\left(12\cdot\frac{-L+ \sqrt{L^2-96(Q-n^*)}}{48}-18-d_u^*+d_t^*\right)-d_u^*-d_t^*+12,
\end{align*}
which simplifies to
\begin{equation}\label{eq:e^*}
    e^*=6n^*-\sqrt{96n^*+L^2-96Q}+L+3d_u^*-5d_t^*+84.
\end{equation}

We start by classifying which values of \(\mu^*_i\) and \(\tau^*_i\) correspond to \(12\)-gons and have a sum less than \(18\). Since \(k=\max\{u_1,\dots, u_6\}\) and \(P^*\) has a rotationally symmetry, we may assume \(\mu_1=0\). For each case we compute \(e^*\) in terms of \(n^*\) using \eqref{eq:e^*} and verify that \(e^*\leq 6n^*-\sqrt{96n^*-63}\) unless \(\mu^*_i=0\) and \(\tau^*_i=2\) \((i=1,\dots,6)\), in which case \(e^*= 6n^*-4\sqrt{6n^*-6}\) and by using \eqref{eq:LQ} $n^*=24k^2-24k+7$. Algorithm \ref{alg:IS} is the pseudocode for this computation. The Python implementation can be found attached to this arXiv submission under the file name \texttt{Inductive\_step.py}.

\begin{algorithm}
    \caption{Algorithm for outstanding cases in inductive step}\label{alg:IS}
    \begin{algorithmic}
        \State cases \(\gets [(\mu_1^*,\mu_2^*,\dots, \mu_6^*,\tau_1^*,\tau_2^*,\dots, \tau_6^*) \text{ satisfying \ref{con:11}}, \ref{con:12}, \ref{con:13} \text{ and }  \ref{con:14}]\)
        \begin{enumerate}
            \item \(\mu_i^*,\tau_i^*\geq 0\), \((i=1,\dots, 6)\) and \(\mu_1^*=0\) \label{con:11}
            \item \(\sum_{i=1}^6 \mu_i^*+\tau_i^*<18\)\label{con:12}
            \item \(0=-\mu_1^*+\mu_4^*+\tau_1^*-\tau_4^*-\tau_2^*+\tau_5^*+\mu_3^*-\mu_6^*-2\tau_3^*+2\tau_6^*\)\label{con:13}
            \item \(0=\tau_1^*-\tau_4^*-\mu_2^*+\mu_5^*+2\tau_2^*-2\tau_5^*-\mu_3^*+\mu_6^*+\tau_3^*-\tau_6^*\)\label{con:14}
        \end{enumerate}
        \Comment{Conditions \ref{con:11} and \ref{con:12} come from Claim \ref{cla:maxmin} and Claim \ref{cla:bbig}.}

        \Comment{Conditions \ref{con:13} and \ref{con:14} come from Lemma \ref{lem:uiticlose}.}
        \For{\((\mu_1^*,\mu_2^*,\dots, \mu_6^*,\tau_1^*,\tau_2^*,\dots, \tau_6^*)\) in cases}
            \State \(L\gets L(\tau^*_1,\tau^*_2,\tau^*_3,\tau^*_4,\tau^*_5,\tau^*_6,\mu^*_2,\mu^*_3,\mu^*_4,\mu^*_6)\) \Comment{Comes from equation \eqref{eq:LQ}.}
            \State \(Q\gets Q(\tau^*_1,\tau^*_2,\tau^*_3,\tau^*_4,\tau^*_5,\tau^*_6,\mu^*_2,\mu^*_3,\mu^*_4,\mu^*_6)\)\Comment{Comes from equation \eqref{eq:LQ}.}
            \State \(d_u^*\gets \sum_{i=1}^6 \mu_i^*\)
            \State \(d_t^*\gets \sum_{i=1}^6 \tau_i^*\)
            \State edge\_formula \(\gets 6n^*-\sqrt{96n^*+L^2-96Q}+L+3d_u^*-5d_t^*+84 \) 
            \State \Comment{Comes from equation \eqref{eq:e^*}.}
            \If{edge\_formula \(>6n^*-\sqrt{96n^*-63}\)}
            { \Return\((\mu_1^*,\dots, \mu_6^*,\tau_1^*,\dots, \tau_6^*)\)}
            \EndIf
        \EndFor
    \end{algorithmic}
\end{algorithm}

\subsection{Proof of Upper Bound of Theorem \ref{the:e}}
We aim to prove Theorem \ref{the:e} by induction \(n\). Clearly it is enough to show Theorem \ref{the:e} for all induced subgraphs of \(\Lambda_U\). The base cases are when \(n\leq 7\), which are easily verified. For the inductive step, suppose \(n\geq 8\) and \(\Lambda_U[S]\) is an \(n\)-vertex subgraph of \(\Lambda_U\) with \(e\) edges. Additionally, suppose \(\Lambda_U[S]\) has the maximum number of edges out of all \(n\)-vertex subgraphs of \(\Lambda_U\). The inductive hypothesis is that all \(n'\)-vertex subgraphs of \(\Lambda_U\), where \(3\leq n'<n\), have at most \(e(n')\) edges.

\begin{claim}
    We may assume \(\Lambda_U[S]\) is \(2\)-connected.
\end{claim}
\begin{proof}
    \(\Lambda_U[S]\) must be connected, otherwise we could translate a connected component of \(\Lambda_U[S]\) to form additional edges contradicting the maximality of \(\Lambda_U[S]\). Suppose \(\Lambda_U[S]\) had a cut vertex \(v\). When we remove \(v\) from \(\Lambda_U[S]\) we create two connected components \(G_1=\Lambda_U[S_1]\) and \(G_2=\Lambda_U[S_2]\) with \(n_1\) and \(n_2\) vertices, and \(e_1\) and \(e_2\) edges.
    
    \textbf{Case 1: \(n_1 < 4\) or \(n_2 < 4\).}
    Without loss of generality suppose \(n_1\leq 3\). If \(n_1=1\), there is \(1\) edge removed when deleting \(S_1\) from \(\Lambda_U[S]\). Applying the inductive hypothesis to the remaining graph we obtain \(e\leq e(n-1)+1\leq e(n)\).
    So supposing that \(n_1\in \{2,3\}\), then there are at most \(6\) edges removed when deleting \(G_1\) from \(\Lambda_U\). Then the inductive hypothesis implies
    \begin{align*}
    e\leq 6(n-n_1)-4\sqrt{6(n-n_1)-6}+6
    \leq 6n-\sqrt{96n-63} \text{ for }n_1=2,3 \text{ if }n\geq 6.
    \end{align*}
    
    \textbf{Case 2: \(n_1, n_2\geq 4\).} Applying the inductive hypothesis to \(G_1\) and \(G_2\) we obtain
    \begin{align*}
    e&\leq 6n_1-4\sqrt{6n_1-6}+6n_2-4\sqrt{6n_2-6}+\deg(v)
    \\&\leq 6n-6-4\sqrt{6(n-5)-6}-4\sqrt{18}+\deg(v).
    \end{align*}
    We can bound the degree of \(v\) by noticing the neighbourhood of \(v\) must be disconnected. It is easy to see through Menger's theorem that the neighbourhood of \(v\) in \(\Lambda_U\) is \(4\)-connected, which implies that the degree of \(v\) in \(\Lambda_U[S]\) is at most \(8\), hence
    \[
    e\leq6n-6-4\sqrt{6(n-5)-6}-4\sqrt{24}+8\leq 6n-\sqrt{96n-63} \text{ for }n\geq 7. \qedhere
    \]
    
\end{proof}

\begin{claim}\label{cla:noline}
    We may assume there is no line parallel to an element in \(U\) intersecting \(\Lambda\) but disjoint from $S$, and with vertices from $S$ on either side of the line.
\end{claim}
\begin{proof}
    Suppose there is such a line \(L\). Since \(\Lambda_U[S]\) is \(2\)-connected there must be at least two edges of \(\Lambda_U[S]\) that cross \(L\). If \(L\) is parallel to a short edge (say parallel to $g_1$), the only edges of \(\Lambda_U[S]\) that can cross \(L\) will be long edges perpendicular to \(L\) (in direction $g_1-2g_2$). We shift a set of vertices on one side of the line towards \(L\) along one of the directions in $U$. This is depicted in Figure~\ref{fig:Unitshift}, where the part of $S$ above $L$ is shifted by $-g_2$. Notice that each edge that crossed \(L\) is maintained after this shift, so the edges of \(\Lambda_U[S]\) can only increase. This process must stop after a finite number of shifts as the area of \(\Lambda_U[S]\) strictly decreases.

    \begin{figure}
    \centering
    \begin{subfigure}[t]{0.4\textwidth}
\begin{tikzpicture}[>=Stealth,scale=0.7]
\def\points{(e1),(e2),($(e1)+(e2)$),($2*(e2)-(e1)$),($(e2)-(e1)$), ($-2*(e1)+(e2)$),($-1*(e1)$),($-1*(e2)$),($-1*(e1)-(e2)$),($-2*(e2)+(e1)$),($-1*(e2)+(e1)$), ($2*(e1)-(e2)$)}
\def\bpoints{(e2),($-1*(e1)+(e3)$),($(e1)+(e2)$),($(e3)+(e2)$),($(e1)-(e2)$),($-1*(e1)-(e2)$),($-1*(e1)-(e3)-(e2)$), ($-1*(e2)$)}
\coordinate (e1) at (1,0);
\coordinate (e2) at (60:1);
\coordinate (e3) at (120:1);
\begin{scope}[on background layer]
\clip (-3.1,-2) rectangle (2.6,2);
\foreach \x in {-5,...,5} 
    {
    \foreach \y in {-3,...,3}
    {
    \coordinate (cur) at ($\x*(e1)+\y*(e2)$);
    \filldraw[draw=black] (cur) circle [radius=0.025];
    }
    }
\draw[red!60!white, dashed] ($-5*(e1)$)--($5*(e1)$);
\end{scope}

\foreach \p in \bpoints 
    {
    \filldraw[draw=black] \p circle [radius=0.05];
    }
\draw (e2)--($(e1)+(e2)$);
\draw (e2)--($2*(e2)-(e1)$);
\draw ($2*(e2)-(e1)$)--($(e1)+(e2)$);
\draw ($(e2)$)--($(e1)-(e2)$);
\draw ($(e3)-(e1)$)--($2*(e2)-(e1)$);
\draw ($(e3)-(e1)$)--($-1*(e2)-(e1)$);
\draw ($-2*(e2)$)--($-1*(e2)-(e1)$);
\draw ($-2*(e2)$)--($-1*(e3)$);
\draw ($-1*(e2)$)--($-1*(e2)+(e1)$);
\draw ($-1*(e2)$)--($-1*(e2)-(e1)$);
\draw ($-1*(e2)$)--($-2*(e2)$);

\end{tikzpicture}
        \caption{\(L\) parallel to a short edge}
    \end{subfigure}
    \hfill
    \begin{subfigure}[t]{0.4\textwidth}
\begin{tikzpicture}[>=Stealth,scale=0.7]
\def\points{(e1),(e2),($(e1)+(e2)$),($2*(e2)-(e1)$),($(e2)-(e1)$), ($-2*(e1)+(e2)$),($-1*(e1)$),($-1*(e2)$),($-1*(e1)-(e2)$),($-2*(e2)+(e1)$),($-1*(e2)+(e1)$), ($2*(e1)-(e2)$)}

\def\bpoints{(0,0),($-1*(e1)+(e3)-(e2)$),($(e1)$),($(e3)$),($(e1)-(e2)$),($-1*(e1)-(e2)$),($-1*(e1)-(e3)-(e2)$), ($-1*(e2)$)}
\coordinate (e1) at (1,0);
\coordinate (e2) at (60:1);
\coordinate (e3) at (120:1);
\begin{scope}[on background layer]
\clip (-3.1,-2) rectangle (2.6,2);
\foreach \x in {-5,...,5} 
    {
    \foreach \y in {-3,...,3}
    {
    \coordinate (cur) at ($\x*(e1)+\y*(e2)$);
    \filldraw[draw=black] (cur) circle [radius=0.025];
    }
    }
\draw[red!60!white, dashed] ($-5*(e1)$)--($5*(e1)$);
\end{scope}

\foreach \p in \bpoints 
    {
    \filldraw[draw=black] \p circle [radius=0.05];
    }
\draw (0,0)--($(e1)$);
\draw (0,0)--($(e2)-(e1)$);
\draw ($(e2)-(e1)$)--($(e1)$);
\draw ($(0,0)$)--($-1*(e3)$);
\draw ($-2*(e1)$)--($(e2)-(e1)$);
\draw ($-2*(e1)$)--($-1*(e2)-(e1)$);
\draw ($-2*(e2)$)--($-1*(e2)-(e1)$);
\draw ($-2*(e2)$)--($-1*(e3)$);
\draw ($(0,0)$)--($-1*(e3)-2*(e1)$);
\draw ($(e1)$)--($-1*(e3)$);
\draw ($-1*(e2)$)--($-1*(e2)+(e1)$);
\draw ($-1*(e2)$)--($-1*(e2)-(e1)$);
\draw ($-1*(e2)$)--($-2*(e2)$);
\draw ($-1*(e2)$)--($(e1)$);
\draw ($-1*(e2)$)--($(0,0)$);
\draw ($-1*(e2)$)--($-2*(e1)$);
\draw ($-1*(e2)$)--($(e3)$);
\end{tikzpicture}
        \caption{Translate vertices above \(L\) by \(-g_2\)}
    \end{subfigure}
    \caption{Process depicting shift when \(L\) is parallel to a short edge}
    \label{fig:Unitshift}
    \end{figure}
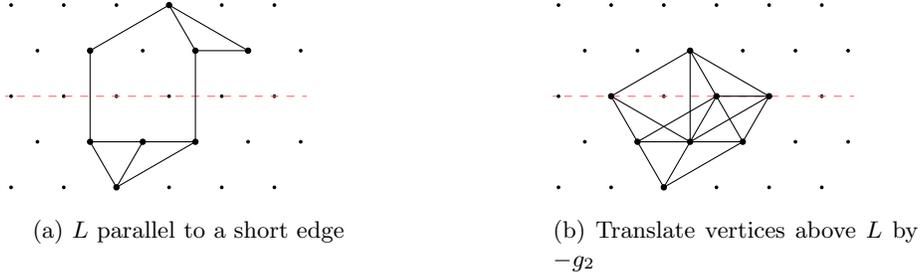

    Now suppose \(L\) is parallel to a long edge (say parallel to $2g_1-g_2$ as in Figure~\ref{subfig}). We apply a unit shift perpendicular to \(L\) (direction $-g_2$ in Figure~\ref{subfigb}) to the set of vertices on one side of \(L\). After this shift vertices may overlap and edges may be lost. We claim there are at most \(4m\) edges lost during this shift, where \(m\) is the number of vertices that overlap. We will show this by deleting edges before the shift so that the remaining edges are not lost after the shift.

    A pair of vertices in \(S\) overlap if they lie on different sides of $L$ and are connected by a short edge perpendicular to $L$. There are three types of edges we need to delete. The first type are edges connecting pairs of vertices that overlap. The second type is when a pair of vertices, $v_1$ and $v_2$ (Figure~\ref{subfig}), overlap and a vertex $v$ is adjacent to both of them (for example, $v_3$ or $v_4$ in Figure~\ref{subfig}). Since after the shift the edges $v_1v$ and $v_2v$ overlap, we will delete one of these edges. The third type is when $v_1'$ and $v_2'$ are another pair of vertices that overlap and are adjacent to $v_1$ and $v_2$ respectively, by a long edge parallel to $L$. In this case the edges $v_1v_1'$ and $v_2v_2'$ overlap after the shift so we will delete one of them.

    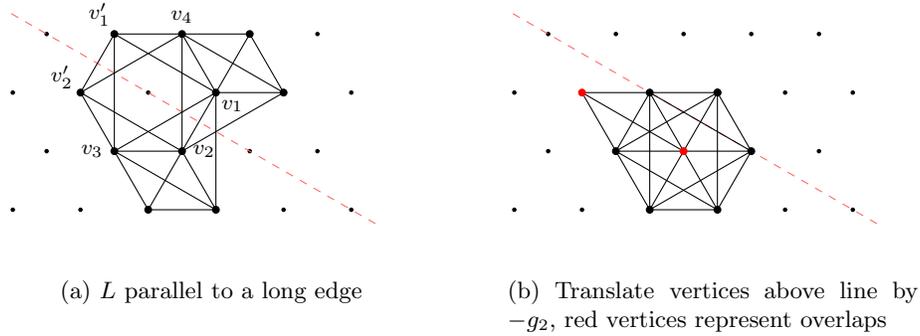
\begin{figure}
    \centering
    \begin{subfigure}[t]{0.45\textwidth}
\begin{tikzpicture}[>=Stealth,scale=0.9]
\def\points{(e1),(e2),($(e1)+(e2)$),($2*(e2)-(e1)$),($(e2)-(e1)$), ($-2*(e1)+(e2)$),($-1*(e1)$),($-1*(e2)$),($-1*(e1)-(e2)$),($-2*(e2)+(e1)$),($-1*(e2)+(e1)$), ($2*(e1)-(e2)$)}

\def\bpoints{(0,0),($-1*(e1)+(e3)-(e2)$),($(e1)$),($(e3)$),($(e3)-(e1)$),($-1*(e1)-(e2)$),($-1*(e1)-(e3)-(e2)$), ($-1*(e2)$), (e2), ($-1*(e3)-(e2)$)}
\coordinate (e1) at (1,0);
\coordinate (e2) at (60:1);
\coordinate (e3) at (120:1);
\begin{scope}[on background layer]
\clip (-3.1,-2.5) rectangle (2.4,1.5);
\foreach \x in {-5,...,5} 
    {
    \foreach \y in {-3,...,3}
    {
    \coordinate (cur) at ($\x*(e1)+\y*(e2)$);
    \filldraw[draw=black] (cur) circle [radius=0.025];
    }
    }
\draw[red!60!white, dashed] ($-1*(e1)-5*(e1)+5*(e3)$)--($-1*(e1)+5*(e1)-5*(e3)$);
\end{scope}

\foreach \p in \bpoints 
    {
    \filldraw[draw=black] \p circle [radius=0.05];
    }
\draw (0,0)--($(e1)$);
\draw (0,0)--($(e2)-(e1)$);
\draw ($(e2)-(e1)$)--($(e1)$);
\draw ($-2*(e1)$)--($(e2)-(e1)$);
\draw ($-2*(e1)$)--($-1*(e2)-(e1)$);
\draw ($-2*(e2)$)--($-1*(e2)-(e1)$);
\draw ($(0,0)$)--($-1*(e3)-2*(e1)$);
\draw ($-1*(e2)$)--($-1*(e2)-(e1)$);
\draw ($-1*(e2)$)--($-2*(e2)$);
\draw ($-1*(e2)$)--($(e1)$);
\draw ($-1*(e2)$)--($(0,0)$);
\draw ($-1*(e2)$)--($-2*(e1)$);
\draw ($-1*(e2)$)--($(e3)$);
\draw ($-1*(e1)+(e3)$)--($(e3)$);
\draw ($-1*(e1)+(e3)$)--($-2*(e1)$);
\draw ($-1*(e1)+(e3)$)--($-1*(e2)-(e1)$);
\draw ($-1*(e1)+(e3)$)--(0,0);
\draw ($(e2)$)--($1*(e1)$);
\draw ($(e2)$)--($1*(e3)$);
\draw ($(e2)$)--(0,0);
\draw ($-1*(e2)-(e3)$)--($-2*(e2)$);
\draw ($-1*(e2)-(e3)$)--($-1*(e2)$);
\draw ($-1*(e2)-(e3)$)--($-1*(e2)-(e1)$);
\draw ($-1*(e2)-(e3)$)--(0,0);
\node at (0,0) [label={[label distance=-2mm]330:\footnotesize{$v_1$}}] {};
\node at ($-1*(e2)$) [label={[label distance=-1mm]-2:\footnotesize{$v_2$}}] {};
\node at ($-2*(e1)$) [label={[label distance=-2mm]110:\footnotesize{$v_2'$}}] {};
\node at ($-1*(e1)+(e3)$) [label={[label distance=-2mm]170:\footnotesize{$v_1'$}}] {};
\node at ($(e3)$) [label={[label distance=-1mm]90:\footnotesize{$v_4$}}] {};
\node at ($-1*(e1)-(e2)$) [label={[label distance=-1mm]180:\footnotesize{$v_3$}}] {};

\end{tikzpicture}
        \caption{\(L\) parallel to a long edge}
        \label{subfig}
    \end{subfigure}
    \hfill
    \begin{subfigure}[t]{0.45\textwidth}

\begin{tikzpicture}[>=Stealth,scale=0.9]
\def\points{(e1),(e2),($(e1)+(e2)$),($2*(e2)-(e1)$),($(e2)-(e1)$), ($-2*(e1)+(e2)$),($-1*(e1)$),($-1*(e2)$),($-1*(e1)-(e2)$),($-2*(e2)+(e1)$),($-1*(e2)+(e1)$), ($2*(e1)-(e2)$)}
\def\bpoints{,($-1*(e3)$),($-1*(e1)$),($-1*(e1)-(e2)$),($-1*(e1)-(e3)-(e2)$), (0,0), ($-1*(e3)-(e2)$)}
\coordinate (e1) at (1,0);
\coordinate (e2) at (60:1);
\coordinate (e3) at (120:1);
\begin{scope}[on background layer]
\clip (-3.1,-2.5) rectangle (2.4,1.5);
\foreach \x in {-5,...,5} 
    {
    \foreach \y in {-3,...,3}
    {
    \coordinate (cur) at ($\x*(e1)+\y*(e2)$);
    \filldraw[draw=black] (cur) circle [radius=0.025];
    }
    }
\draw[red!60!white, dashed] ($-1*(e1)-5*(e1)+5*(e3)$)--($-1*(e1)+5*(e1)-5*(e3)$);
\foreach \p in \bpoints 
    {
    \filldraw[draw=black] \p circle [radius=0.05];
    }
\draw ($-1*(e1)$)--($-2*(e1)$);
\draw ($-1*(e1)$)--($-1*(e1)-(e2)$);
\draw ($-1*(e1)$)--($-1*(e3)$);
\draw ($-1*(e1)$)--($-2*(e2)$);
\draw ($-1*(e2)$)--($-1*(e1)$);
\draw ($-2*(e1)$)--($-1*(e2)-(e1)$);
\draw ($-1*(e2)$)--($-2*(e1)$);
\draw ($-1*(e2)$)--($-1*(e3)$);
\draw ($-1*(e2)$)--($-1*(e2)-(e1)$);
\draw ($-1*(e2)$)--($-2*(e2)$);
\draw ($-1*(e2)-(e1)$)--($-2*(e2)$);
\draw ($-1*(e2)-(e3)$)--($-2*(e2)$);
\draw ($-1*(e2)-(e3)$)--($-1*(e2)$);
\draw ($-1*(e2)-(e3)$)--($-1*(e2)-(e1)$);
\draw ($-1*(e2)-(e3)$)--(0,0);
\draw ($-1*(e1)$)--(0,0);
\draw ($-1*(e2)$)--(0,0);
\draw ($-1*(e3)$)--(0,0);
\draw ($-1*(e2)-(e3)$)--($-1*(e3)$);
\draw ($-2*(e2)$)--($-1*(e3)$);
\draw ($-1*(e2)-(e1)$)--(0,0);
\end{scope}
\filldraw[fill=red,draw=red] ($-1*(e2)$) circle [radius=0.05];
\filldraw[fill=red,draw=red] ($-2*(e1)$) circle [radius=0.05];
\end{tikzpicture}
        \caption{Translate vertices above line by \(-g_2\), red vertices represent overlaps}
        \label{subfigb}
    \end{subfigure}
    \caption{Process depicting shift when \(L\) is parallel to a long edge}
    \label{fig:sqrt3shift}
    \end{figure}

    For each pair of vertices that overlap we delete one edge of the first type, which implies $m$ edges of this type are deleted. A pair of vertices that overlap have at most $4$ vertices adjacent to both of them. Each of these vertices correspond to one edge that must be deleted, which implies at most $4$ edges of this type are deleted per overlap. It is a simple case analysis to confirm in the case $4$ edges of this type are deleted at an overlap, there are at least two new edges created after the shift. Similarly, if $3$ edges of this type are deleted, then at least $1$ new edge will be created after the shift. In all cases, in total we only lose at most $2$ edges of this type per overlap, implying $2m$ edges of this type are lost after the shift. Since each pair of vertices that overlap has at most two other overlapping pairs connected by long edges parallel to $L$, it is easily seen that at most $m$ edges of the third type are lost.

    All together this shows that at most $4m$ edges are lost during this shift. If \(m=0\) then all edges that cross \(L\) are maintained, and since the area of \(\Lambda_U[S]\) strictly decreases, this process must stop after a finite number of shifts. If \(m\geq 1\) we apply the inductive hypothesis to the resulting graph after the shift and obtain 
    \begin{align*}
    e&\leq 6(n-m)-4\sqrt{6(n-m)-6}+4m
    \\&\leq 6n-\sqrt{96n-63} \text{ for all } n\geq \max\{2m,10\}.
    \end{align*}

    If no edges of the third type are lost, we lose at most $3m$ edges from the shift. Applying the inductive hypothesis to the resulting graph we obtain:
    \begin{align*}
    e&\leq 6(n-m)-4\sqrt{6(n-m)-6}+3m
    \\&\leq 6n-\sqrt{96n-63} \text{ for all } n\geq \max\{2m,7\}.
    \end{align*}

    To finish the proof when $L$ is parallel to a long edge, it remains to be shown: if there are edges lost of the third type (implying $m\geq 2$), and $n=8$ or $9$, then $e\leq e(n)$. Since in this case there are edges lost of the third type there exists two overlapping pairs of vertices $v_1,v_2$ and $v_1',v_2'$ that are connected by long edges parallel to $L$. This implies $v_1$ is not adjacent to $v_2'$ and $v_2$ is not adjacent to $v_1'$. It is easily seen from Figure \ref{subfig} there are exactly two elements of $\Lambda\setminus{L}$, say  $v_3$ and $v_4$, that are adjacent to all $v_1,v_2,v_1'$ and $v_2'$. Every other element of $\Lambda\setminus{L}$ is adjacent to at most two of $v_1,v_2,v_1'$ and $v_2'$. 

    We distinguish between three cases:
    
    \textbf{\(S\) contains $v_3$ and $v_4$.} Then $v_1,v_2,v_3,v_2',v_1'$ and $v_4$ form the vertices of a regular unit hexagon. Place a vertex at the centre of the hexagon creating at least $6$ new edges. Since $n+1<12$ there is a vertex with degree at most $5$ which when deleted results in an $n$-vertex graph with more edges then \(\Lambda_U[S]\), so \(\Lambda_U[S]\) was not extremal.

    \textbf{\(S\) does not contain $v_3$ and $v_4$.} This implies in the case $n=8$ that $e\leq \binom{8}{2}-2-4\times 2=18 <e(8)$. In the case $n=9$ we obtain $e\leq \binom{9}{2}-2-5\times2=24<e(n)$.

    \textbf{\(S\) contains $v_3$ or $v_4$ but not both.} Without loss of generality we may assume it contains $v_3$. From Figure \ref{subfig} it is easily verified that in $\Lambda\setminus{L\cup \{v_4\}}$ there are only three vertices adjacent with three out of the five vertices $v_1,v_2,v_1',v_2'$ and $v_3$. The rest are adjacent with at most two. This implies in the case $n=8$ that $e\leq \binom{8}{2}-2-3\times 2=20 <e(8)$ and in the case $n=9$ we obtain $e\leq \binom{9}{2}-2-3\times 2-3=25 =e(9)$.

\end{proof}

\begin{claim}\label{cla:defG}
    $\partial(\operatorname{hull}(S))\leq \partial(S)$.
\end{claim}
\begin{proof}
    To show this, we create an injection from the set \(B_1=\{ uv\in E(\Lambda_U): u\in \operatorname{hull}(S), v\notin \operatorname{hull}(S)\}\) to \(B_2=\{ uv\in E(\Lambda_U): u\in S, v\notin S\}\). Consider an element \(u\in \operatorname{hull}(S)\) with \(uv\in E(\Lambda_U)\) and \(v\notin \operatorname{hull}(S)\). Let \(L\) be the line containing this edge. Since the boundary of \(\Lambda_U[\operatorname{hull}(S)]\) is convex, this line will contain exactly two edges in \(B_1\).
    By Claim \ref{cla:noline}, \(L\) must intersect a vertex of \(\Lambda_U\). This implies \(L\) must contain at least two edges in \(B_2\).
    
    The injection will map the two edges contained in \(L\) in the set \(B_1\) to some choice of two edges contained in \(L\) in the set \(B_2\). This map will be injective as no two disjoint lines each containing an element in \(B_1\) have an intersection that contains an element in \(B_2\).
\end{proof}

\begin{claim}\label{cla:neqP}
    If \(\Lambda_U[S]\notin \mathscr{P}\) then \(e\leq 6n-\sqrt{96n-63}\) 
\end{claim}
\begin{proof}
    Suppose \(\Lambda_U[S]\notin \mathscr{P}\) which implies \(\Lambda_U[S]\neq \Lambda_U[\operatorname{hull}(S)]\). Suppose that \(\Lambda_U[\operatorname{hull}(S)]\) has \(n'\) vertices and \(e'\) edges. By Theorem \ref{the:P} we have \(e'\leq 6n'-4\sqrt{6n'-6} \) which implies by \eqref{eq:eindef} that \(\partial(\operatorname{hull}(S))\geq 8\sqrt{6n'-6}\). By Claim \ref{cla:defG} we have \(\partial(\operatorname{hull}(S))\leq \partial(S)\) which implies
    \begin{align*}
    e&=6n-\frac{\partial(S)}{2}
    \leq 6n-\frac{\partial(\operatorname{hull}(S))}{2}
     \leq 6n-4\sqrt{6n'-6} 
    \\& \leq 6n-4\sqrt{6(n+1)-6}
     =6n-4\sqrt{6n}
     <6n-\sqrt{96n-63}. \qedhere
    \end{align*}
    
\end{proof}

Theorem \ref{the:P} and Claim \ref{cla:neqP} completes the proof of Theorem \ref{the:e}.

\section{Existence of an ordering}\label{sec:seq}

In Section \ref{sec:upbound} we proved that the only \(12\)-gon in \(\mathscr{P}\) that has \(6n-4\sqrt{6n-6}\) edges when \(n=24k^2-24k+7\) for some \(k\geq 2\), is where \(u_i=k\) and \(t_i=k-1\) \((i=1,\dots,6)\). Furthermore, we showed that any subgraph of \(\Lambda_U\) not in \(\mathscr{P}\) has at most \(\lfloor6n-\sqrt{96n-63}\rfloor\) edges. Since \(6n-4\sqrt{6n-6}>\lfloor6n-\sqrt{96n-63}\rfloor\) for \(n=24k^2-24k+7\) for all \(k\geq 2\), at these special values of \(n\), the \(12\)-gon with \(u_i=k\) and \(t_i=k-1\) \((i=1,\dots,6)\) is the unique extremal graph. 

To obtain an ordering of \(\Lambda\) that attains \(e(n)\) edges for each \(n\), we start off by giving an ordering of the first \(55\) points constituting the unique extremal \(12\)-gon for which \(u_i=2\) and \(t_i=1\) \((i=1,\dots, 6)\). Then we give, for any \(k\geq 3\), an ordering from the (unique) extremal \(12\)-gon with \(u_i=k-1\) and \(t_i=k-2\) \((i=1,\dots,6)\), to the next \(12\)-gon with \(u_i=k\) and \(t_i=k-1\) \((i=1,\dots,6)\). For any \(k\geq 3\), we call the former \(12\)-gon the \emph{initial} \(12\)-gon, and the latter \(12\)-gon the \emph{terminal} \(12\)-gon.

Figure~\ref{fig:basecaseextremal} shows an ordering of \(55\) vertices of \(\Lambda_U\) that, for each \(n\), attain \(e(n)\) many edges from the subgraph of \(\Lambda_U\) induced by the first \(n\) terms. One can verify this using Table \ref{tableBC}.

\begin{figure}
    \centering
\begin{tikzpicture}
\def\points{(e1),(e2),($(e1)+(e2)$),($2*(e2)-(e1)$),($(e2)-(e1)$), ($-2*(e1)+(e2)$),($-1*(e1)$),($-1*(e2)$),($-1*(e1)-(e2)$),($-2*(e2)+(e1)$),($-1*(e2)+(e1)$), ($2*(e1)-(e2)$)}
\def\bpoints{($3*(e2)+(e3)-(e1)$),($3*(e2)+(e3)$),($3*(e2)+(e1)$),($3*(e2)+(e1)-(e3)$),($3*(e1)+(e2)$),($3*(e1)-(e3)$),($3*(e1)-(e3)-(e2)$),($-3*(e3)+(e1)$),($-3*(e3)-(e2)$),($-3*(e3)-(e2)-(e1)$),($-3*(e2)-(e3)$),($-3*(e2)-(e1)$),($-3*(e2)-(e1)+(e3)$),($-3*(e1)-(e2)$),($-3*(e1)+(e3)$),($-3*(e1)+(e3)+(e2)$),($3*(e3)-(e1)$),($3*(e3)+(e2)$)}
\def\ipoints{(0,0), (e1), (e2), (e3), ($-1*(e1)$), ($-1*(e2)$), ($-1*(e3)$),
            ($(e1)+(e1)$), ($(e1)+(e2)$), ($(e1)-(e3)$),
            ($-1*(e1)-(e1)$), ($-1*(e1)-(e2)$), ($-1*(e1)+(e3)$),
            ($(e2)+(e2)$), ($(e2)+(e3)$), ($(e2)+(e3)-(e1)$),
            ($-1*(e2)-(e2)$), ($-1*(e2)-(e3)$), ($-1*(e2)-(e3)+(e1)$)}
\coordinate (e1) at (1,0);
\coordinate (e2) at (60:1);
\coordinate (e3) at (120:1);
\begin{scope}[on background layer]
\clip ($3*(e2)+(e3)$)-- ($3*(e2)+(e1)$) -- ($3*(e1)+(e2)$) --($3*(e1)-(e3)$) -- ($-3*(e3)+(e1)$) -- ($-3*(e3)-(e2)$) -- ($-3*(e2)-(e3)$) -- ($-3*(e2)-(e1)$) -- ($-3*(e1)-(e2)$) -- ($-3*(e1)+(e3)$) -- ($3*(e3)-(e1)$) -- ($3*(e3)+(e2)$);
\filldraw[fill=gray!10!white] ($3*(e2)+(e3)$)-- ($3*(e2)+(e1)$) -- ($3*(e1)+(e2)$) --($3*(e1)-(e3)$) -- ($-3*(e3)+(e1)$) -- ($-3*(e3)-(e2)$) -- ($-3*(e2)-(e3)$) -- ($-3*(e2)-(e1)$) -- ($-3*(e1)-(e2)$) -- ($-3*(e1)+(e3)$) -- ($3*(e3)-(e1)$) -- ($3*(e3)+(e2)$) -- cycle;
\foreach \x in {-5,...,5} 
    {
    \foreach \y in {-5,...,5}
    {
    \coordinate (cur) at ($\x*(e1)+\y*(e2)$);
    \filldraw[draw=black] (cur) circle [radius=0.05];
    }
    }
\end{scope}
\foreach \p in \bpoints 
    {
    \filldraw[draw=black] \p circle [radius=0.05];
    }
\node at ($2*(e3)$) [label={[label distance=-1mm]90:\footnotesize{$1$}}] {};
\node at ($(e3)+(e2)$) [label={[label distance=-1mm]90:\footnotesize{$2$}}] {};
\node at ($(e3)$) [label={[label distance=-1mm]90:\footnotesize{$3$}}] {};
\node at ($(e3)-(e1)$) [label={[label distance=-1mm]90:\footnotesize{$4$}}] {};
\node at ($2*(e3)-(e1)$) [label={[label distance=-1mm]90:\footnotesize{$5$}}] {};
\node at ($3*(e3)$) [label={[label distance=-1mm]90:\footnotesize{$6$}}] {};
\node at ($2*(e3)+(e2)$) [label={[label distance=-1mm]90:\footnotesize{$7$}}] {};
\node at ($2*(e2)+(e3)$) [label={[label distance=-1mm]90:\footnotesize{$8$}}] {};
\node at ($2*(e2)$) [label={[label distance=-1mm]90:\footnotesize{$9$}}] {};
\node at ($(e2)$) [label={[label distance=-1mm]90:\footnotesize{$10$}}] {};
\node at (0,0) [label={[label distance=-1mm]90:\footnotesize{$11$}}] {};
\node at ($-1*(e1)$) [label={[label distance=-1mm]90:\footnotesize{$12$}}] {};
\node at ($(e1)+(e2)$) [label={[label distance=-1mm]90:\footnotesize{$13$}}] {};
\node at ($(e1)$) [label={[label distance=-1mm]90:\footnotesize{$14$}}] {};
\node at ($-1*(e2)$) [label={[label distance=-1mm]90:\footnotesize{$15$}}] {};
\node at ($-2*(e1)$) [label={[label distance=-1mm]90:\footnotesize{$16$}}] {};
\node at ($-2*(e1)+(e3)$) [label={[label distance=-1mm]90:\footnotesize{$17$}}] {};
\node at ($-1*(e2)-(e1)$) [label={[label distance=-1mm]90:\footnotesize{$18$}}] {};
\node at ($-1*(e3)$) [label={[label distance=-1mm]90:\footnotesize{$19$}}] {};
\node at ($-1*(e2)-(e3)$) [label={[label distance=-1mm]90:\footnotesize{$20$}}] {};
\node at ($-1*(e3)+(e1)$) [label={[label distance=-1mm]90:\footnotesize{$21$}}] {};
\node at ($2*(e1)$) [label={[label distance=-1mm]90:\footnotesize{$22$}}] {};
\node at ($2*(e2)+(e1)$) [label={[label distance=-1mm]90:\footnotesize{$23$}}] {};
\node at ($3*(e2)$) [label={[label distance=-1mm]90:\footnotesize{$24$}}] {};
\node at ($2*(e2)+2*(e3)$) [label={[label distance=-1mm]90:\footnotesize{$25$}}] {};
\node at ($3*(e3)+(e2)$) [label={[label distance=-1mm]90:\footnotesize{$26$}}] {};
\node at ($-1*(e1)+3*(e3)$) [label={[label distance=-1mm]90:\footnotesize{$27$}}] {};
\node at ($-2*(e1)+2*(e3)$) [label={[label distance=-1mm]90:\footnotesize{$28$}}] {};
\node at ($-3*(e1)$) [label={[label distance=-1mm]90:\footnotesize{$29$}}] {};
\node at ($-2*(e1)+-1*(e2)$) [label={[label distance=-1mm]90:\footnotesize{$30$}}] {};
\node at ($-2*(e2)$) [label={[label distance=-1mm]90:\footnotesize{$31$}}] {};
\node at ($-2*(e3)$) [label={[label distance=-1mm]90:\footnotesize{$32$}}] {};
\node at ($2*(e1)+(e2)$) [label={[label distance=-1mm]90:\footnotesize{$33$}}] {};
\node at ($2*(e1)-(e3)$) [label={[label distance=-1mm]90:\footnotesize{$34$}}] {};
\node at ($-2*(e3)+(e1)$) [label={[label distance=-1mm]90:\footnotesize{$35$}}] {};
\node at ($-2*(e3)-1*(e2)$) [label={[label distance=-1mm]90:\footnotesize{$36$}}] {};
\node at ($-2*(e2)-1*(e1)$) [label={[label distance=-1mm]90:\footnotesize{$37$}}] {};
\node at ($-2*(e2)-1*(e3)$) [label={[label distance=-1mm]90:\footnotesize{$38$}}] {};
\node at ($3*(e1)$) [label={[label distance=-1mm]90:\footnotesize{$39$}}] {};
\node at ($2*(e2)+2*(e1)$) [label={[label distance=-1mm]90:\footnotesize{$40$}}] {};
\node at ($3*(e2)+(e1)$) [label={[label distance=-1mm]90:\footnotesize{$41$}}] {};
\node at ($3*(e2)+(e3)$) [label={[label distance=-1mm]90:\footnotesize{$42$}}] {};
\node at ($-3*(e3)$) [label={[label distance=-1mm]90:\footnotesize{$43$}}] {};
\node at ($2*(e1)-2*(e3)$) [label={[label distance=-1mm]90:\footnotesize{$44$}}] {};
\node at ($3*(e1)-(e3)$) [label={[label distance=-1mm]90:\footnotesize{$45$}}] {};
\node at ($3*(e1)+(e2)$) [label={[label distance=-1mm]90:\footnotesize{$46$}}] {};
\node at ($-3*(e3)+(e1)$) [label={[label distance=-1mm]90:\footnotesize{$47$}}] {};
\node at ($-3*(e3)-(e2)$) [label={[label distance=-1mm]90:\footnotesize{$48$}}] {};
\node at ($-3*(e2)$) [label={[label distance=-1mm]90:\footnotesize{$49$}}] {};
\node at ($-2*(e2)-2*(e3)$) [label={[label distance=-1mm]90:\footnotesize{$50$}}] {};
\node at ($-2*(e2)-2*(e1)$) [label={[label distance=-1mm]90:\footnotesize{$51$}}] {};
\node at ($-3*(e1)+(e3)$) [label={[label distance=-1mm]90:\footnotesize{$52$}}] {};
\node at ($-3*(e1)-(e2)$) [label={[label distance=-1mm]90:\footnotesize{$53$}}] {};
\node at ($-3*(e2)-(e1)$) [label={[label distance=-1mm]90:\footnotesize{$54$}}] {};
\node at ($-3*(e2)-(e3)$) [label={[label distance=-1mm]90:\footnotesize{$55$}}] {};
\end{tikzpicture}
    \caption{The first \(55\) terms in the ordering of \(\Lambda\)}
    \label{fig:basecaseextremal}
\end{figure}
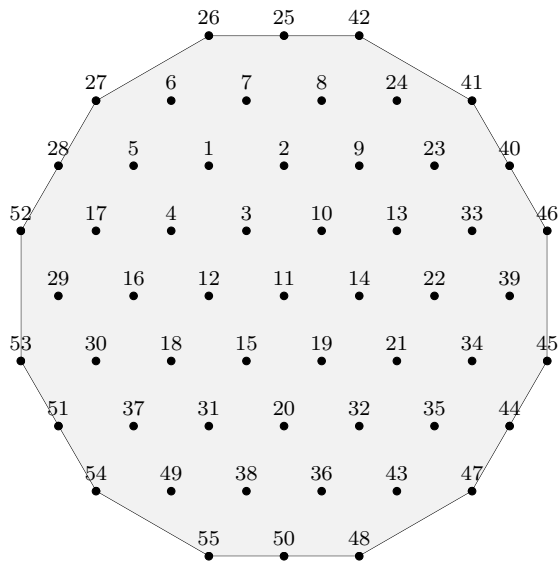

From now on we assume \(k\geq 3\), and give a construction of a way to add vertices starting with the initial \(12\)-gon, that builds up to the terminal \(12\)-gon. The sequence of vertices we add generates a sequence of how many edges we add per vertex. This will have to be a sequence of \(5\)'s and \(6\)'s since \(e(n)-e(n-1)=5,6\) for all \(n\geq 56\). This means that if the graph we have built so far has \(n\) vertices and \(e(n)\) edges, and we add a vertex which creates \(6\) edges, then this (\(n+1\))-vertex graph has at least \(e(n+1)\) edges.
In fact, the graph will have exactly \(e(n+1)\) edges as we have already established the upper bound. So for the sequence of vertices we construct, we only need to justify that the graph has \(e(n)\) many edges when we obtained that graph by adding a vertex with degree~\(5\). 

Note that the ordering we construct must be greedy for it to have \(e(n)\) many edges at each step.
When we add a vertex to a side of a \(12\)-gon, we create \(5\) edges. If we want to add another vertex to this graph, the only way in which we can add \(6\) edges, is by placing this vertex next to the previously added vertex. This continues until the whole side is completely filled.

Thus it is sufficient to find a sequence of sides from the initial \(12\)-gon to fill up in a way that builds up to the terminal \(12\)-gon. The only vertices added that create \(5\) edges are the ones added to a full \(12\)-gon. Let \(f(n)\) be the number of edges in terms of \(n\) for such a \(12\)-gon. It is then sufficient to show that \(f(n)+5\geq e(n+1)\).

If \(f(n)=6n-\sqrt{96n+a}\) then it is enough to show \(a<33\) as
\[
 6n-\sqrt{96n+a}+5\geq \lfloor6(n+1)-\sqrt{96(n+1)-63}\rfloor\]
\[\iff \lceil\sqrt{96n+33}\rceil- \sqrt{96n+a}\geq 1.
\]

Since \(\sqrt{96n+a}\) is an integer, the latter is equivalent to 
\[
\sqrt{96n+33}>\sqrt{96n+a}\iff 33>a.
\]
So if we can find a sequence of \(12\)-gons, starting and ending at the initial and terminal \(12\)-gon, where each term in the sequence has an edge formula in the form \(6n-\sqrt{96n+a}\) where \(a<33\), and each subsequent term adds a complete side from the previous term, then we would have obtained the desired construction. 

When adding a side to a \(12\)-gon, we can add onto a \(t_i\) or a \(u_i\) for some \(i=1,\dots, 6\). When we add onto a \(t_i\) side, we decrease $t_i$ by \(1\) and increase the adjacent \(u_i\) and \(u_{i+1}\) sides by \(1\). When we add onto a \(u_i\) side, we decrease it by \(3\) and increase the adjacent \(t_{i-1}\) and \(t_i\) sides by \(1\). Table~\ref{table} presents such a sequence of sides with confirmation that their edge-number formulas satisfy \(a<33\). (Note that since \(k\geq 3\), all sides in the construction are non-negative.)

We found this sequence of length $48$ presented in Table~\ref{table} by constructing an auxiliary directed graph where each node represents a \(12\)-gon. First, we create a node representing the initial \(12\)-gon, and then we generate the neighbourhood of this node to be all \(12\)-gons obtained from the initial \(12\)-gon with a side added onto it, and continuing in this manner until we reach the terminal \(12\)-gon. We only add \(12\)-gons with parameter \(a<33\) to the auxiliary graph. We then apply a standard Breadth-First Search algorithm to this graph to find a path from the initial to the terminal \(12\)-gon. The auxiliary directed graph has $1152$ vertices and $2550$ edges. The Python implementation can be found attached to this arXiv submission under the file name \texttt{Sequence\_solutions.py}.

\begin{table}[ht]
\centering\small
\resizebox{\columnwidth}{!}{
\begin{tabular}{lllllllllllllc}
  \hline
  \;\;$u_1$&\;\;$u_2$&\;\;$u_3$&\;\;$u_4$&\;\;$u_5$&\;\;$u_6$&\;\;$t_1$&\;\;$t_2$&\;\;$t_3$&\;\;$t_4$&\;\;$t_5$&\;\;$t_6$&\quad\; \(f(n)\) & Side\;\; \\ 
  \hline
  \hline
  $k - 1$&$k - 1$&$k - 1$&$k - 1$&$k - 1$&$k - 1$&$k - 2$&$k - 2$&$k - 2$&$k - 2$&$k - 2$&$k - 2$&$6 n - \sqrt{96 n - 96}$& \\ 
\hline
$k$&$k$&$k - 1$&$k - 1$&$k - 1$&$k - 1$&$k - 3$&$k - 2$&$k - 2$&$k - 2$&$k - 2$&$k - 2$&$6 n - \sqrt{96 n - 47}$&$t_1 $\\ 
\hline
$k$&$k + 1$&$k$&$k - 1$&$k - 1$&$k - 1$&$k - 3$&$k - 3$&$k - 2$&$k - 2$&$k - 2$&$k - 2$&$6 n - \sqrt{96 n + 4}$&$t_2 $\\ 
\hline
$k$&$k - 2$&$k$&$k - 1$&$k - 1$&$k - 1$&$k - 2$&$k - 2$&$k - 2$&$k - 2$&$k - 2$&$k - 2$&$6 n - \sqrt{96 n - 39}$&$u_2 $\\ 
\hline
$k + 1$&$k - 1$&$k$&$k - 1$&$k - 1$&$k - 1$&$k - 3$&$k - 2$&$k - 2$&$k - 2$&$k - 2$&$k - 2$&$6 n - \sqrt{96 n + 16}$&$t_1 $\\ 
\hline
$k - 2$&$k - 1$&$k$&$k - 1$&$k - 1$&$k - 1$&$k - 2$&$k - 2$&$k - 2$&$k - 2$&$k - 2$&$k - 1$&$6 n - \sqrt{96 n - 23}$&$u_1 $\\ 
\hline
$k - 1$&$k - 1$&$k$&$k - 1$&$k - 1$&$k$&$k - 2$&$k - 2$&$k - 2$&$k - 2$&$k - 2$&$k - 2$&$6 n - \sqrt{96 n - 60}$&$t_6 $\\ 
\hline
$k - 1$&$k$&$k + 1$&$k - 1$&$k - 1$&$k$&$k - 2$&$k - 3$&$k - 2$&$k - 2$&$k - 2$&$k - 2$&$6 n - \sqrt{96 n + 1}$&$t_2 $\\ 
\hline
$k - 1$&$k$&$k - 2$&$k - 1$&$k - 1$&$k$&$k - 2$&$k - 2$&$k - 1$&$k - 2$&$k - 2$&$k - 2$&$6 n - \sqrt{96 n - 32}$&$u_3 $\\ 
\hline
$k - 1$&$k$&$k - 1$&$k$&$k - 1$&$k$&$k - 2$&$k - 2$&$k - 2$&$k - 2$&$k - 2$&$k - 2$&$6 n - \sqrt{96 n - 63}$&$t_3 $\\ 
\hline
$k$&$k + 1$&$k - 1$&$k$&$k - 1$&$k$&$k - 3$&$k - 2$&$k - 2$&$k - 2$&$k - 2$&$k - 2$&$6 n - \sqrt{96 n + 4}$&$t_1 $\\ 
\hline
$k$&$k - 2$&$k - 1$&$k$&$k - 1$&$k$&$k - 2$&$k - 1$&$k - 2$&$k - 2$&$k - 2$&$k - 2$&$6 n - \sqrt{96 n - 23}$&$u_2 $\\ 
\hline
$k$&$k - 1$&$k$&$k$&$k - 1$&$k$&$k - 2$&$k - 2$&$k - 2$&$k - 2$&$k - 2$&$k - 2$&$6 n - \sqrt{96 n - 48}$&$t_2 $\\ 
\hline
$k + 1$&$k$&$k$&$k$&$k - 1$&$k$&$k - 3$&$k - 2$&$k - 2$&$k - 2$&$k - 2$&$k - 2$&$6 n - \sqrt{96 n + 25}$&$t_1 $\\ 
\hline
$k - 2$&$k$&$k$&$k$&$k - 1$&$k$&$k - 2$&$k - 2$&$k - 2$&$k - 2$&$k - 2$&$k - 1$&$6 n - \sqrt{96 n + 4}$&$u_1 $\\ 
\hline
$k - 1$&$k$&$k$&$k$&$k - 1$&$k + 1$&$k - 2$&$k - 2$&$k - 2$&$k - 2$&$k - 2$&$k - 2$&$6 n - \sqrt{96 n - 15}$&$t_6 $\\ 
\hline
$k - 1$&$k$&$k$&$k$&$k - 1$&$k - 2$&$k - 2$&$k - 2$&$k - 2$&$k - 2$&$k - 1$&$k - 1$&$6 n - \sqrt{96 n - 32}$&$u_6 $\\ 
\hline
$k - 1$&$k$&$k$&$k$&$k$&$k - 1$&$k - 2$&$k - 2$&$k - 2$&$k - 2$&$k - 2$&$k - 1$&$6 n - \sqrt{96 n - 47}$&$t_5 $\\ 
\hline
$k$&$k$&$k$&$k$&$k$&$k$&$k - 2$&$k - 2$&$k - 2$&$k - 2$&$k - 2$&$k - 2$&$6 n - \sqrt{96 n - 60}$&$t_6 $\\ 
\hline
$k + 1$&$k + 1$&$k$&$k$&$k$&$k$&$k - 3$&$k - 2$&$k - 2$&$k - 2$&$k - 2$&$k - 2$&$6 n - \sqrt{96 n + 25}$&$t_1 $\\ 
\hline
$k - 2$&$k + 1$&$k$&$k$&$k$&$k$&$k - 2$&$k - 2$&$k - 2$&$k - 2$&$k - 2$&$k - 1$&$6 n - \sqrt{96 n + 16}$&$u_1 $\\ 
\hline
$k - 2$&$k - 2$&$k$&$k$&$k$&$k$&$k - 1$&$k - 1$&$k - 2$&$k - 2$&$k - 2$&$k - 1$&$6 n - \sqrt{96 n + 9}$&$u_2 $\\ 
\hline
$k - 2$&$k - 1$&$k + 1$&$k$&$k$&$k$&$k - 1$&$k - 2$&$k - 2$&$k - 2$&$k - 2$&$k - 1$&$6 n - \sqrt{96 n + 4}$&$t_2 $\\ 
\hline
$k - 2$&$k - 1$&$k - 2$&$k$&$k$&$k$&$k - 1$&$k - 1$&$k - 1$&$k - 2$&$k - 2$&$k - 1$&$6 n - \sqrt{96 n + 1}$&$u_3 $\\ 
\hline
$k - 2$&$k - 1$&$k - 1$&$k + 1$&$k$&$k$&$k - 1$&$k - 1$&$k - 2$&$k - 2$&$k - 2$&$k - 1$&$6n- \sqrt{96n}$&$t_3 $\\ 
\hline
$k - 2$&$k - 1$&$k - 1$&$k - 2$&$k$&$k$&$k - 1$&$k - 1$&$k - 1$&$k - 1$&$k - 2$&$k - 1$&$6 n - \sqrt{96 n + 1}$&$u_4 $\\ 
\hline
$k - 2$&$k - 1$&$k - 1$&$k - 1$&$k + 1$&$k$&$k - 1$&$k - 1$&$k - 1$&$k - 2$&$k - 2$&$k - 1$&$6 n - \sqrt{96 n + 4}$&$t_4 $\\ 
\hline
$k - 2$&$k - 1$&$k - 1$&$k - 1$&$k - 2$&$k$&$k - 1$&$k - 1$&$k - 1$&$k - 1$&$k - 1$&$k - 1$&$6 n - \sqrt{96 n + 9}$&$u_5 $\\ 
\hline
$k - 2$&$k - 1$&$k - 1$&$k - 1$&$k - 1$&$k + 1$&$k - 1$&$k - 1$&$k - 1$&$k - 1$&$k - 2$&$k - 1$&$6 n - \sqrt{96 n + 16}$&$t_5 $\\ 
\hline
$k - 2$&$k - 1$&$k - 1$&$k - 1$&$k - 1$&$k - 2$&$k - 1$&$k - 1$&$k - 1$&$k - 1$&$k - 1$&$k$&$6 n - \sqrt{96 n + 25}$&$u_6 $\\ 
\hline
$k - 1$&$k - 1$&$k - 1$&$k - 1$&$k - 1$&$k - 1$&$k - 1$&$k - 1$&$k - 1$&$k - 1$&$k - 1$&$k - 1$&$6 n - \sqrt{96 n - 60}$&$t_6 $\\ 
\hline
$k$&$k$&$k - 1$&$k - 1$&$k - 1$&$k - 1$&$k - 2$&$k - 1$&$k - 1$&$k - 1$&$k - 1$&$k - 1$&$6 n - \sqrt{96 n - 47}$&$t_1 $\\ 
\hline
$k$&$k + 1$&$k$&$k - 1$&$k - 1$&$k - 1$&$k - 2$&$k - 2$&$k - 1$&$k - 1$&$k - 1$&$k - 1$&$6 n - \sqrt{96 n - 32}$&$t_2 $\\ 
\hline
$k$&$k - 2$&$k$&$k - 1$&$k - 1$&$k - 1$&$k - 1$&$k - 1$&$k - 1$&$k - 1$&$k - 1$&$k - 1$&$6 n - \sqrt{96 n - 15}$&$u_2 $\\ 
\hline
$k + 1$&$k - 1$&$k$&$k - 1$&$k - 1$&$k - 1$&$k - 2$&$k - 1$&$k - 1$&$k - 1$&$k - 1$&$k - 1$&$6 n - \sqrt{96 n + 4}$&$t_1 $\\ 
\hline
$k - 2$&$k - 1$&$k$&$k - 1$&$k - 1$&$k - 1$&$k - 1$&$k - 1$&$k - 1$&$k - 1$&$k - 1$&$k$&$6 n - \sqrt{96 n + 25}$&$u_1 $\\ 
\hline
$k - 1$&$k - 1$&$k$&$k - 1$&$k - 1$&$k$&$k - 1$&$k - 1$&$k - 1$&$k - 1$&$k - 1$&$k - 1$&$6 n - \sqrt{96 n - 48}$&$t_6 $\\ 
\hline
$k - 1$&$k$&$k + 1$&$k - 1$&$k - 1$&$k$&$k - 1$&$k - 2$&$k - 1$&$k - 1$&$k - 1$&$k - 1$&$6 n - \sqrt{96 n - 23}$&$t_2 $\\ 
\hline
$k - 1$&$k$&$k - 2$&$k - 1$&$k - 1$&$k$&$k - 1$&$k - 1$&$k$&$k - 1$&$k - 1$&$k - 1$&$6 n - \sqrt{96 n + 4}$&$u_3 $\\ 
\hline
$k - 1$&$k$&$k - 1$&$k$&$k - 1$&$k$&$k - 1$&$k - 1$&$k - 1$&$k - 1$&$k - 1$&$k - 1$&$6 n - \sqrt{96 n - 63}$&$t_3 $\\ 
\hline
$k$&$k + 1$&$k - 1$&$k$&$k - 1$&$k$&$k - 2$&$k - 1$&$k - 1$&$k - 1$&$k - 1$&$k - 1$&$6 n - \sqrt{96 n - 32}$&$t_1 $\\ 
\hline
$k$&$k - 2$&$k - 1$&$k$&$k - 1$&$k$&$k - 1$&$k$&$k - 1$&$k - 1$&$k - 1$&$k - 1$&$6 n - \sqrt{96 n + 1}$&$u_2 $\\ 
\hline
$k$&$k - 1$&$k$&$k$&$k - 1$&$k$&$k - 1$&$k - 1$&$k - 1$&$k - 1$&$k - 1$&$k - 1$&$6 n - \sqrt{96 n - 60}$&$t_2 $\\ 
\hline
$k + 1$&$k$&$k$&$k$&$k - 1$&$k$&$k - 2$&$k - 1$&$k - 1$&$k - 1$&$k - 1$&$k - 1$&$6 n - \sqrt{96 n - 23}$&$t_1 $\\ 
\hline
$k - 2$&$k$&$k$&$k$&$k - 1$&$k$&$k - 1$&$k - 1$&$k - 1$&$k - 1$&$k - 1$&$k$&$6 n - \sqrt{96 n + 16}$&$u_1 $\\ 
\hline
$k - 1$&$k$&$k$&$k$&$k - 1$&$k + 1$&$k - 1$&$k - 1$&$k - 1$&$k - 1$&$k - 1$&$k - 1$&$6 n - \sqrt{96 n - 39}$&$t_6 $\\ 
\hline
$k - 1$&$k$&$k$&$k$&$k - 1$&$k - 2$&$k - 1$&$k - 1$&$k - 1$&$k - 1$&$k$&$k$&$6 n - \sqrt{96 n + 4}$&$u_6 $\\ 
\hline
$k - 1$&$k$&$k$&$k$&$k$&$k - 1$&$k - 1$&$k - 1$&$k - 1$&$k - 1$&$k - 1$&$k$&$6 n - \sqrt{96 n - 47}$&$t_5 $\\ 
\hline
$k$&$k$&$k$&$k$&$k$&$k$&$k - 1$&$k - 1$&$k - 1$&$k - 1$&$k - 1$&$k - 1$&$6 n - \sqrt{96 n - 96}$&$t_6 $\\ 
\hline
\end{tabular}
}

\caption{Building up to the next \(12\)-gon}\label{table}
\end{table}

\section{Conclusion}
We presented two examples related to the question of Barber and Erde \cite{BE2018} of whether there is always a nested sequence of optimal solutions in a Cayley graph $\mathbb{Z}^d_U$.

Our first example (Theorem~\ref{the:e}) shows that $d=1$ is special in the sense that the positive result of Briggs and Wells \cite{BW2024} cannot be extended to higher dimensions.
In this example the generating set has vectors that are positive multiples of each other. It would be interesting to find an example where this does not happen.

Our second example (Theorem~\ref{the:P}) is a positive result in dimension~$2$, and is proved using some computation.
The proof method has some potential to be generalized to other special cases, and it would be worth exploring this direction.
It seems that many more examples, both positive and negative, will be needed to understand for which $U$ there exists a nested sequence.

\section*{Acknowledgment}
We thank Peter Allen for his insightful suggestions.

\bibliographystyle{plain}
\bibliography{main}
\end{document}